\documentclass[preprint,3p,12pt]{elsarticle}

\usepackage{amsmath,amssymb,amsfonts,amsthm}
\usepackage{graphicx}
\usepackage{subfigure}
\usepackage{float}

\usepackage{booktabs}
\usepackage{colortbl}
\usepackage{caption}
\usepackage{epstopdf}
\usepackage{algorithm}
\usepackage{algpseudocode}
\usepackage{array}
\usepackage{longtable}

\usepackage{bm}
\usepackage{makecell}
\usepackage{multirow}
\usepackage{multicol}

\newtheorem{theorem}{Theorem}[section]

\newtheorem{proposition}[theorem]{Proposition}

\usepackage{lineno}
\modulolinenumbers[1]

\def\bn{\mathbf{n}}
\def\bG{\mathbf{G}}

\def\fvec#1{\boldsymbol{#1}} 
\def\bvec#1{\mathbf{#1}} 
\def\bmat#1{\mathbf{#1}} 

\def\NIT{\texttt{NIT}}

\def\TT{$\text{T}_{\text{con}}^{\text{T}}$}
\def\TC{$\text{T}_{\text{con}}^{\text{C}}$}
\def\TA{$\text{T}_{\text{con}}^{\hat{\text{A}}}$}
\def\TO{$\text{T}_{\text{con}}^{\text{O}}$}
\def\Tcon{$\text{T}_{\text{con}}$}
\def\Tsol{$\text{T}_{\text{sol}}$}
\def\Ttot{$\text{T}_{\text{tot}}$}

\journal{Computers \& Mathematics with Applications}

\begin{document}


\begin{frontmatter}




\title{A DOFs condensation based algorithm for solving saddle point systems in contact computation \footnote{This work was funded by National Natural Science Foundation of China (No. 12171045)} }


\author[a,b]{Xiaoyu Duan}
\ead{duanxiaoyu19@gscaep.ac.cn}
\author[a,c]{Hengbin An\corref{cor1}}
\ead{an\_hengbin@iapcm.ac.cn}
\author[a,c]{Zeyao Mo}
\ead{ zeyao\_mo@iapcm.ac.cn}
\cortext[cor1]{Corresponding author.}

\affiliation[a]{organization={Institute of Applied Physics and Computational Mathematics},
            city={Beijing},
            postcode={100094}, 
            country={China}}
            
\affiliation[b]{organization={Graduate School of China Academy of Engineering Physics},
            city={Beijing},
            postcode={100088}, 
            country={China}}

\affiliation[c]{organization={CAEP Software Center for High Performance Numerical Simulation},
            city={Beijing},
            postcode={100088}, 
            country={China}}
            

\begin{abstract}
In contact mechanics computation, the constraint conditions on the contact surfaces are typically enforced by the Lagrange multiplier method, resulting in a saddle point system. The mortar finite element method is usually employed to discretize the variational form on the meshed contact surfaces, leading to a large-scale discretized saddle point system. Due to the indefiniteness of the discretized system, it is a challenge to solve the saddle point algebraic system. 
For two-dimensional tied contact problem, an efficient DOFs condensation technique is developed. The essential of the proposed method is to carry out the DOFs elimination by using the tridiagonal characteristic of the mortar matrix.
The scale of the linear system obtained after DOFs elimination is smaller,  and the matrix is symmetric positive definite. 
By using the preconditioned conjugate gradient (PCG) method, 
the linear system can be solved efficiently.
Numerical results show the effectiveness of the method.
\end{abstract}




\end{frontmatter}



\section{Introduction}
\label{sect:intrd}

The contact phenomena of structural mechanics is widespread 
in real engineering applications, such as automobile design and manufacturing, aero-engine design and optimization, etc~\cite{an2022shear, INTERNODES2022}.
The finite element method (FEM) is widely used in numerical simulation of contact problems~\cite{NTShughes1976finite, wriggers1995finite}.
For numerical computation of the contact problem of structural mechanics, 
two aspects related to contact should be considered. The first is the discretization technique for contact surfaces; the second is the enforcement of the contact constraint to the variational form that is used in FEM.
The essential of contact discretization is the process of searching and matching the adjacent relationship of cells' entities 
(nodes, edges or faces) on contact surfaces after meshing the contact bodies.

The discretization approach of the contact surface mainly includes two categories: the node-to-surface (NTS) approach~\cite{NTShughes1976finite, NTShughes1977finite, NTSbathe1985solution}, and the surface-to-surface (STS) approach~\cite{STSsimo1985perturbed, STSpapadopoulos1992mixed}.
In the NTS method, the continuity constraints are enforced at 
discrete nodal points. In contrast, in the STS method, the continuity constraints are formulated along the contact boundary in a weak integral form. This makes STS  
more accurate than NTS when the meshes are non-matching on contact surfaces~\cite{MORpuso2004mortar}. 
And one typical STS approach is the mortar finite element method that was first introduced as a domain decomposition technique in 1994~\cite{MORbernardi1994new}. The mortar method 
has one advantage to address non-matching surface meshes flexibly and it is considered as a state-of-the-art method for computational contact mechanics~\cite{MORbelgacem1999mortar, MORmcdevitt2000mortar, MORtur2009mortar, popp2012mortar}.

Another aspect aside from the discretization approach is the enforcement of contact constraint conditions.
Two representative methods of enforcing the contact constraint conditions 
are penalty method~\cite{PENALTYzavarise1992real, PENALTYzavarise2009modified} and  
the Lagrange multiplier method~\cite{LAGbelytschko1991contact, LAGpapadopoulos1998lagrange, LAGfranceschini2020algebraically}.
Although the penalty method has an advantage that the size of the discrete linear system is the same as the case 
without contact, this method requires a user-defined penalty parameter affecting the convergence and stability of the numerical solution. 
It is often hard to select an appropriate penalty parameter when multiple materials and contact pairs are involved in the numerical simulation of structural mechanics. 
If the penalty parameter is improperly chosen, it can lead to a severely ill-conditioned linear system, making it challenging to solve.
The Lagrange multiplier method is an exact constraint enforcement strategy where the Lagrange multiplier has a physical interpretation as the contact force acting on the surface. 
Therefore, imposing contact constraints using the Lagrange multiplier method is more natural. 
However, the Lagrange multiplier method introduces an additional Lagrange multiplier degree of freedom and leads to an undesirable increase in the size of the corresponding linear system. 
This system exhibits a typical saddle point structure, which poses significant challenges for solvers and preconditioners, so an efficient solution method is needed.

To avoid solving saddle point linear systems in contact problems, Wohlmuth et al. developed the dual mortar finite element method~\cite{DUALwohlmuth2000mortar,  DUALpopp2012dual, DUALvon2020contact}. 
In the dual mortar method, 
the Lagrange multipliers are constructed using special basis functions based on a biorthogonality with displacement basis functions. 
This specificity can turn a mortar matrix into a diagonal matrix, allowing for easy condensing of the Lagrange multiplier degrees of freedom. 
From a numerical solution perspective, the dual mortar method offers certain advantages. 
This method avoids directly solving a saddle point problem and effectively reduces the size of the system.

Although the dual mortar method can effectively transform and solve saddle point systems, it has two drawbacks compared to the standard mortar method. 
Firstly, the basis functions in the dual mortar method have negative values, which may cause solutions of the discrete systems to violate physical constraints. 
Secondly, the dual mortar method involves bidirectional search computations on both sides of the contact surface during the discrete calculation process.
For large-scale parallel computing, the search cost between 
the contact surfaces is expensive, and this leads to  
a bottleneck for large scale numerical simulation 
of engineering applications. 
Therefore, it is essential to solve the saddle point linear system discretized by the standard mortar method~\cite{reverse2022}. 

Wiesner et al. considered using the AMG (algebraic multigrid) preconditioned Krylov subspace method to solve the saddle point equations in contact problems~\cite{AMGwiesner2021algebraic}. 
They took into account contact constraints on all coarse grid levels.
Specifically, in the aggregation-based coarsening process of AMG, a novel Lagrange multiplier aggregation method for contact surfaces is introduced.
This method ensures that the matrices corresponding to each coarse grid level maintain the saddle point structure, guaranteeing the separation of displacement and Lagrange multiplier degrees of freedom.
Although Wiesner's work has made significant improvements in solving linear systems of equations in contact problems by considering the $2 \times 2$ structure of the global matrix, their work did not further utilize the properties of the mortar matrix in contact problems.

On the other hand, scholars have conducted extensive research on solving general saddle point systems~\cite{SADDELbenzi2005numerical, SADDELaxelsson2015unified, SADDELpestana2015natural}.
Various solution methods and preconditioned techniques have been developed, such as the RHSS method~\cite{HSSbai2017regularized, HSSbai2019regularized}, Uzawa method~\cite{UZAWAho2017accelerating}, Schur complement method, null space method, constraint preconditioned method, approximate incomplete factorization method, and multigrid methods~\cite{SADAMGbenzi2002preconditioning, SADAMGnotay2019convergence}.
In the Schur complement method, the original hard-to-solve saddle point system is converted to two smaller-sized systems, embodying the idea of elimination in its construction.
However, directly applying the Schur complement method to the contact mechanics problem is complex since the inversion of a matrix 
close to the original problem's size is concerned.
Although approaches such as diagonal approximation can reduce the computational cost of the matrix inversion, it is less effective in solving contact problems due to the lack of diagonal dominance in the associated matrix. 

In addition, the Schur complement approximation allows for the construction of different types of block preconditioners, which have been prevalent across various fields
~\cite{PREur2008comparison, PREliu2020nested}.
For example, the SIMPLE method~\cite{patankar1983calculation}, a type of block preconditioner based on the Schur complement, is widely used in computational fluid dynamics.
By introducing the idea of the SIMPLE method to the computation
of contact mechanics, one of the most effective preconditioning methods is constructed~\cite{AMGwiesner2021algebraic}. 
While extensive research has been conducted on solving saddle point problems, efficient solution strategies specifically tailored to contact problems have not received significant attention.

In this paper, we consider solving the saddle point systems
by using the standard mortar approach and the Lagrange multipliers approach in contact problems.
For simplicity of discussion, we focus on the 
case of two-dimensional tied contact problems.
By employing the application characteristics and matrix structure, a novel and efficient solution method 
based on DOFs condensation is proposed for solving the 
saddle point system.
In the proposed method, 
the degrees of freedom of the discrete Lagrange multiplier and the slave side displacement are condensed (eliminated), 
and the original symmetric indefinite saddle point system 
is transformed into a smaller-sized symmetric positive definite system. 
The condensation process is achieved by utilizing the tridiagonal block structure of one of the mortar matrices.
Theoretical analysis demonstrates that the matrix, after elimination, possesses the property of symmetrical positive definiteness.
Numerical results for three contact models show that the obtained system after condensing can be solved 
effectively by the AMG preconditioned conjugate gradient (CG) method
while the original saddle point system is difficult to be 
solved by most of the algorithms in PETSc.
Moreover, the iteration numbers of the CG method are relatively stable as the size of the problem increases,
and this is attractive for large scale 
contact mechanics numerical simulations.

The remainder of this paper is organized as follows. 
Section~\ref{sect:math-model-tie-contact} provides a general description of the two-dimensional tied contact problem.
The mortar method based discretization of the problem and the corresponding discretized linear system is presented in Section~\ref{sect:discrete-linear-system}. 
Section~\ref{sect:method} presents the condensed elimination method for saddle point systems. 
The property of the matrix of the transformed system 
is also provided in this section. 
Numerical results for three models are given in  
Section~\ref{sect:numer}.
Finally, some conclusions and remarks are given in Section~\ref{sect:remark} to end the paper.

\section{Mathematical model for tied contact problem}
\label{sect:math-model-tie-contact}

In this section, the tied contact models, including the differential form and variational form, are introduced.

For tied contact problem, the contact bodies are tied together
by some shared surfaces.
Taking two bodies as an example, two elastic bodies are denoted as $\Omega^{(1)}$ and $\Omega^{(2)}$, where $\Omega^{(1)}$ is a salve body and $\Omega^{(2)}$ is a master body,
and $\Omega^{(i)} \subset \mathbb{R}^d$, $d = 2$ or 3.
In each subdomain $\Omega^{(i)}$, the linear elastic boundary value problem has the following form~\cite{popp2012mortar}
\begin{equation}
\label{strong}
	\begin{aligned}
		-\nabla \cdot \boldsymbol{\sigma}^ {(i)}  &= \fvec{f}^{(i)}  & \; \text{ in } \; \Omega^{(i)}, \\
		\fvec{u}^{(i)} 			  &= \hat{\fvec{u}}^ {(i)} & \;  \text{ on } \; \Gamma_{\mathrm{u}}^{(i)}, \\
		\boldsymbol{\sigma}^{(i)} \cdot \bn^ {(i)} &= \fvec{t}^{(i)} & \; \text{ on } \; \Gamma_{\sigma}^{(i)},
	\end{aligned}
\end{equation}
which consists of the equilibrium equation,
displacement boundary conditions, and stress boundary conditions.
Herein, $\fvec{u}^{(i)}$ is the displacement,
$\boldsymbol{\sigma}^{(i)}$ is the stress that is dependent on the displacement,
$\fvec{f}^{(i)}$ is the body force, $\bn^ {(i)}$ is the unit outer normal vector
of the domain $\Omega^{(i)}$.
$\Gamma_{\mathrm{u}}^{(i)}$ denotes the Dirichlet boundary, where the displacements
are prescribed by $\hat{\fvec{u}}^ {(i)}$, and $\Gamma_{\sigma}^{(i)}$ represents the Neumann boundary, where the tractions are given by $\fvec{t}^{(i)}$.
For the tied contact problem,
the displacement of each body on the contact surface is equal.
In other words, it can be stated as
\begin{equation}
	\fvec{u}^{(1)} = \fvec{u}^{(2)} \quad  \text{on} \quad \Gamma_{\mathrm{c}} ,
\end{equation}
where $\Gamma_{\mathrm{c}}$ represents the contact boundary of
$\Omega^{(1)}$ and $\Omega^{(2)}$.

According to the variational principle, the contact problem can also be equivalently described as a constrained functional minimization problem
\begin{equation}
\label{eqn:constrain-min-prob}
	\left\{
	\begin{aligned}
		\min        \, & \mathcal{W} (\fvec{u}) \\
		\mbox{s.t.} \, & \bG(\fvec{u}) = 0
	\end{aligned}	
	\right.
\end{equation}
where
\begin{eqnarray*}
\mathcal{W} (\fvec{u}) &=& \sum_{i=1}^{2} \left[ \int_{\Omega^{(i)}} \left(\boldsymbol{\sigma}^{(i)} : \boldsymbol{\varepsilon}^{(i)} - \fvec{f}^{(i)} \cdot \fvec{u}^{(i)} \right) \mathrm{d} V
	 -\int_{\Gamma_{\sigma}^{(i)}} \fvec{t}^{(i)} \cdot \fvec{u}^{(i)} \mathrm{d} S \right] , \\
\bG(\fvec{u}) &=& 	\fvec{u}^{(1)} - \fvec{u}^{(2)}.
\end{eqnarray*}

By the Lagrange multiplier method,
the above constrained minimization problem~\eqref{eqn:constrain-min-prob}
can be transformed to an unconstrained minimization problem
\begin{equation}
	\min \mathcal{W}_{\mathrm{LM}} (\fvec{u}, \boldsymbol{\lambda}) := \mathcal{W}(\fvec{u}) + \int_{\Gamma_{\mathrm{c}}} \boldsymbol{\lambda} \cdot \bG(\fvec{u})\text{d} S,
\end{equation}
where $\boldsymbol{\lambda}$ denotes the Lagrange multiplier. Thereby the mixed variational formulation is given by
\begin{equation}
	\delta \mathcal{W}_{\mathrm{LM}}
	=\delta \fvec{u}^{\top} \frac{\partial \mathcal{W}_{\mathrm{LM}} }{\partial \fvec{u}}
        + \delta \boldsymbol{\lambda}^{\top} \frac{\partial \mathcal{W}_{\mathrm{LM}} }{\partial \boldsymbol{\lambda}}
	= 0.
\end{equation}

To define the weak form of the tied contact problem, 
define the following trial function and test function spaces
\begin{equation}
	\begin{aligned}
		\boldsymbol{\mathcal{U}}^{(i)} &= \left\{ \fvec{u}^{(i)} \in \left[ H^{1}(\Omega^{(i)}) \right]^d : \fvec{u}^{(i)} = \hat{\fvec{u}}^{(i)}  \; \text{on} \; \Gamma_{\mathrm{u}} \right\}, \\
		\boldsymbol{\mathcal{V}}^{(i)} &= \left\{ \delta \fvec{u}^{(i)} \in \left[H^{1}(\Omega^{(i)}) \right]^d : \delta \fvec{u}^{(i)}=0 \; \text{on} \; \Gamma_{\mathrm{u}} \right\},
	\end{aligned}
\end{equation}
where $H^{1}(\Omega^{(i)})$ is the Soblev space~\cite{adams2003sobolev}.
Moreover, the Lagrange multiplier is chosen from space $\boldsymbol{\mathcal{M}} = [H^{-\frac{1}{2}}(\Gamma_{\mathrm{c}})]^d$,
the dual space of the trace space 
of $\boldsymbol{\mathcal{V}}^{(1)}$ on $\Gamma_{\mathrm{c}}$.
Based on these considerations, a saddle point type weak formulation is derived:
find $\fvec{u}^{(i)} \in \boldsymbol{\mathcal{U}}^{(i)}$ and $\boldsymbol{\lambda} \in \boldsymbol{\mathcal{M}}$ such that
\begin{equation}
	\label{weakform}
		\begin{aligned}
			-\delta \mathcal{W}_{\mathrm{int}, \mathrm{ext}} \left( \fvec{u}^{(i)}, \delta \fvec{u}^{(i)} \right)
			- \delta \mathcal{W}_{\mathrm{mt}} \left( \boldsymbol{\lambda}, \delta \fvec{u}^{(i)} \right) = 0,
			& \quad \forall \; \delta \fvec{u}^{(i)} \in \boldsymbol{\mathcal{V}}^{(i)}, \\
			\delta \mathcal{W}_{\lambda} \left( \fvec{u}^{(i)}, \delta \boldsymbol{\lambda} \right) = 0,
			& \quad \forall \; \delta \boldsymbol{\lambda} \in \boldsymbol{\mathcal{M}}.
		\end{aligned}
\end{equation}
Herein, $\mathcal{W}_{\mathrm{int},\mathrm{ext}}$ represents the conventional virtual work term due to internal and external forces, $\mathcal{W}_{\mathrm{mt}}$ represents the virtual work term associated with contact forces, and $\mathcal{W}_{\lambda}$ represents the weak form directly related to the tied contact constraint. They are defined as follows:
\begin{eqnarray}
		-\delta \mathcal{W}_{\mathrm{int}, \mathrm{ext}} &=& \sum_{i=1}^{2}\left[ \int_{\Omega^{(i)}} \left( \boldsymbol{\sigma}^{(i)} : \delta \boldsymbol{\varepsilon}^{(i)} - \fvec{f}^{(i)} \cdot \delta \fvec{u}^{(i)} \right) \mathrm{d} V
		-\int_{\Gamma_{\sigma}^{(i)}} \fvec{t}^{(i)} \cdot \delta \fvec{u}^{(i)} \mathrm{d} S \right] ,\\
		-\delta \mathcal{W}_{\mathrm{mt}} &=& \int_{\Gamma_{\mathrm{c}}} \boldsymbol{\lambda} \cdot\left(\delta \fvec{u}^{(1)} - \delta \fvec{u}^{(2)} \right) \mathrm{d} S ,\\
		\quad \, \delta \mathcal{W}_{\lambda} &=& \int_{\Gamma_{\mathrm{c}}} \delta \boldsymbol{\lambda} \cdot \left( \fvec{u}^{(1)} - \fvec{u}^{(2)} \right) \mathrm{d} S.
\end{eqnarray}

\section{Discretization and the linear system}
\label{sect:discrete-linear-system}

\subsection{Discretization based on mortar method}
\label{subsect:discrete}

The discretized linear system can be obtained through the finite element method based on the weak form~\eqref{weakform}.
In numerical simulations of contact mechanics, the mortar finite element method is a popular discretization approach.
This method is used to handle the coupling terms ($\mathcal{W}_{\mathrm{mt}}$ and $\mathcal{W}_{\lambda}$) between the displacement and Lagrange multiplier on contact surfaces.
And the conventional virtual work term $\mathcal{W}_{\mathrm{int},\mathrm{ext}}$ is discretized by the classical finite element method.

In each contact body, the standard linear finite element method~\cite{wang2003} is used to discretize the term $\mathcal{W}_{\mathrm{int},\mathrm{ext}}$ in~\eqref{weakform}.
After discretization, we have
\begin{eqnarray}
	\label{mat1}
	\begin{aligned}
		-\delta \mathcal{W}_{\mathrm{int}, \mathrm{ext},h}
		&=
		\delta \bvec{d}^{\top} \bvec{f}_{\mathrm{int}}(\bvec{d}) - \delta \bvec{d}^{\top} \bvec{f}_{\mathrm{ext}} \\
		&=
		\delta \bvec{d}^{\top} \bmat{K} \bvec{d} - \delta \bvec{d}^{\top} \bvec{f}_{\mathrm{ext}} ,
	\end{aligned}
\end{eqnarray}
where $\bvec{d}$ represents the vector of nodal displacements, $\delta \bvec{d}$ represents the vector
of nodal virtual displacements,
$\bvec{f}_{\mathrm{int}}(\bvec{d})$ and $\bvec{f}_{\mathrm{ext}}$ respectively denote the internal force vector and the external force vector, and $\bmat{K}$ represents stiffness matrix.
For ease of expression, the subscript $\cdot_h$ denotes the corresponding discrete situation.

To discretize $\delta \mathcal{W}_{\mathrm{mt}}$ and $\delta \mathcal{W}_{\lambda}$, it is necessary to define the displacement interpolation and Lagrange multiplier interpolation on the contact surface $\Gamma_{\mathrm{c},h}^{(i)}$.
Taking the two-body tied contact as an example,
Fig.~\ref{mesh} shows the discretized contact surface with non-matching mesh.

\begin{figure}[htbp]
	\centering
	\includegraphics[scale=0.6]{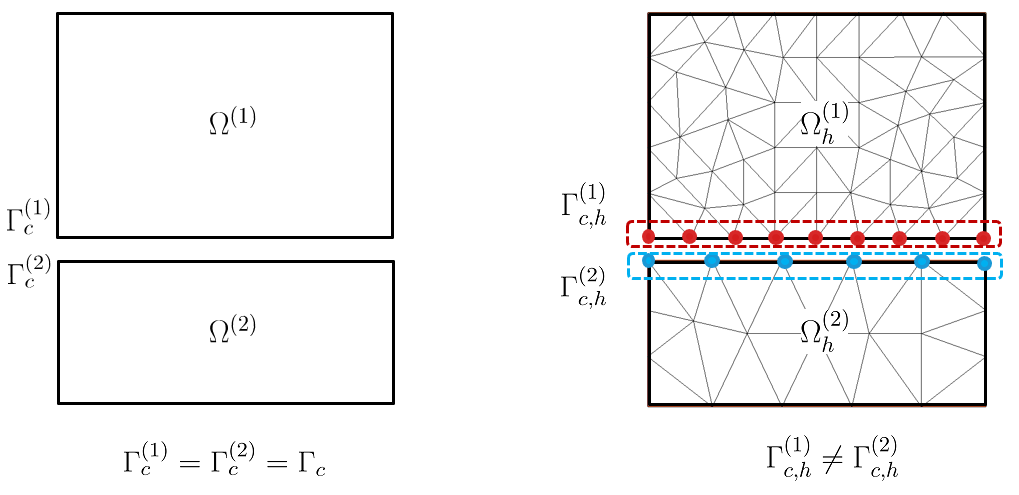}
	\caption{2D two-body tied contact model.}
	\label{mesh}
\end{figure}

The displacement $\fvec{u}_h^{(i)}$ on the disretized contact surface $\Gamma_{\mathrm{c},h}^{(i)}$ is given by
\begin{eqnarray}
	\label{displace}
		\fvec{u}^{(1)}_h = \sum_{k=1}^{n^{(1)}} N_k^{(1)} (\xi^{(1)} ) \bvec{d}_k^{(1)},  \quad
		\fvec{u}^{(2)}_h = \sum_{l=1}^{n^{(2)}} N_l^{(2)} (\xi^{(2)} ) \bvec{d}_l^{(2)}.
\end{eqnarray}
Here, $n^{(1)}$ represents the total number of nodes on $\Gamma_{\mathrm{c},h}^{(1)}$, and $n^{(2)}$ represents the total number of nodes on $\Gamma_{\mathrm{c},h}^{(2)}$.
$\bvec{d}_k^{(1)}$ and $\bvec{d}_l^{(2)}$ respectively represent the displacements of discrete nodes on the slave and the master surface.
The shape functions $N_k^{(1)}$ and $N_l^{(2)}$ represent the one-dimensional linear finite element basis functions.

The Lagrange multiplier is defined on the slave surface, and the corresponding interpolation is:
\begin{eqnarray} \label{lambda}
	\boldsymbol{\lambda}_h = \sum_{j=1}^{m^{(1)}} \Phi_j (\xi^{(1)} ) \boldsymbol{\lambda}_j,
\end{eqnarray}
where $\Phi_j$ is the basis function of the Lagrange multiplier, $j = 1, 2, \ldots, m^{(1)}$.
In most cases, every slave node also acts as a coupling node because it carries the additional Lagrange multiplier.
Therefore, it is common that the number of the Lagrange multipliers is equal to the number of slave surface nodes, that is, $m^{(1)} = n^{(1)}$.
Next, by substituting~\eqref{displace} and~\eqref{lambda} into the weak formulation~\eqref{weakform}, the discrete form of the virtual work $\delta \mathcal{W}_{\mathrm{mt}}$ is expressed as
\begin{eqnarray}
\label{mt}
\begin{aligned}
-\delta \mathcal{W}_{\mathrm{mt}, h}
&=
\sum_{j=1}^{m^{(1)}} \sum_{k=1}^{n^{(1)}} \boldsymbol{\lambda}_{j}^{\top} \left(\int_{\Gamma_{\mathrm{c}, h}^{(1)}} \Phi_{j} N_{k}^{(1)} \mathrm{d} A_{0} \right) \delta \bvec{d}_{k}^{(1)} \\
&-\sum_{j=1}^{m^{(1)}} \sum_{l=1}^{n^{(2)}} \boldsymbol{\lambda}_{j}^{\top} \left( \int_{\Gamma_{\mathrm{c}, h}^{(1)}} \Phi_{j} \left( N_{l}^{(2)} \circ \chi_{h} \right) \mathrm{d} A_{0} \right) \delta \bvec{d}_{l}^{(2)},
\end{aligned}
\end{eqnarray}
where $\chi_{h}: \Gamma_{\mathrm{c},h}^{(1)} \rightarrow \Gamma_{\mathrm{c},h}^{(2)}$ is a discrete mapping from slave side to master side.
Note that $\Gamma_{\mathrm{c}}^{(1)} = \Gamma_{\mathrm{c}}^{(2)}$
in continuous case. However, in discrete case, $\Gamma_{\mathrm{c},h}^{(1)} \neq \Gamma_{\mathrm{c},h}^{(2)}$ because of non-matching mesh on contact surfaces.
Since the numerical integration is only performed on the slave side $\Gamma_{\mathrm{c},h}^{(1)}$,
the map $\chi_{h}$ has a crucial influence on the computation~\cite{popp2012mortar}.

Now define
a $m^{(1)} \times n^{(1)}$ block matrix $\bmat{D}$ and
a $m^{(1)} \times n^{(2)}$ block matrix $\bmat{M}$
by
\begin{eqnarray}
		\begin{aligned}
			&{\bmat{D}}[j, k] = D_{j k} {\bmat{I}}_{\textsf{ndim}},  &j=1, \ldots, m^{(1)}, \quad k= 1, \dots , n^{(1)}\\
			&{\bmat{M}}[j, l] = M_{j l} {\bmat{I}}_{\textsf{ndim}},  &j=1, \ldots, m^{(1)}, \quad l = 1, \dots , n^{(2)}.
		\end{aligned}
\end{eqnarray}
where
\begin{eqnarray*}
D_{j k} = \int_{\Gamma_{\mathrm{c}, h}^{(1)}} \Phi_{j} N_{k}^{(1)} \mathrm{d} S,
\quad
M_{j l} = \int_{\Gamma_{\mathrm{c}, h}^{(1)}} \Phi_{j} \left( N_{l}^{(2)} \circ \chi_{h} \right) \mathrm{d} S,
\end{eqnarray*}
and ${\bmat{I}}_{\textsf{ndim}}$ is the ${\textsf{ndim}} \times {\textsf{ndim}}$ identity matrix.
In general, both $\bmat{D}$ and $\bmat{M}$ are not square matrices.
However, $\bmat{D}$ becomes a square matrix for the common choice $m^{(1)} = n^{(1)}$.

Divide the grid nodes into two sets: one is the node set of the slave body, which includes the slave contact surface node set $\mathcal{S}$ and the internal node set $\mathcal{N}_1$; the other is the node set of the master body, which includes the master contact surface node set $\mathcal{M}$ and the internal node set $\mathcal{N}_2$.
Thus, all displacement degrees of freedom can be represented in a block form as
\begin{equation}
    \bvec{d} = [\bvec{d}_1, \bvec{d}_2], \quad
    \bvec{d}_1 = [\bvec{d}_{\mathcal{N}_1}, \bvec{d}_\mathcal{S}], \quad
    \bvec{d}_2 = [\bvec{d}_{\mathcal{N}_2}, \bvec{d}_\mathcal{M}].
\end{equation}
Correspondingly, by the definition of the matrices
$\bmat{D}$ and $\bmat{M}$,
equation~\eqref{mt} can be written in matrix form:
\begin{eqnarray}
\label{mat2}
-\delta \mathcal{W}_{\mathrm{mt}, h}
			&=& \delta \bvec{d}_\mathcal{S}^{\top} {\bmat{D}}^{\top} \boldsymbol{\lambda} - \delta \bvec{d}_\mathcal{M}^{\top} {\bmat{M}}^{\top} \boldsymbol{\lambda}
			= \delta \bvec{d}^{\top}
			\left[
			\begin{array}{c}
				\mathbf{0} \\ {\bmat{D}}^{\top} \\ \mathbf{0} \\ -{\bmat{M}}^{\top}
			\end{array}
			\right] \boldsymbol{\lambda}  \nonumber \\
			&=& \delta \bvec{d}_1^{\top} \tilde{{\bmat{D}}}^{\top} \boldsymbol{\lambda} - \delta \bvec{d}_2^{\top} \tilde{{\bmat{M}}}^{\top} \boldsymbol{\lambda}
			= \delta \bvec{d}^{\top}
			\left[
			\begin{array}{c}
				\tilde{{\bmat{D}}}^{\top} \\ -\tilde{{\bmat{M}}}^{\top}
			\end{array} \right] \boldsymbol{\lambda}
			= \delta \bvec{d}^{\top} {\bmat{G}}^{\top} \boldsymbol{\lambda} ,
\end{eqnarray}
where
\begin{equation}
   {\bmat{G}} = [\tilde{{\bmat{D}}}, -\tilde{{\bmat{M}}}], \quad
    \tilde{{\bmat{D}}} = [\mathbf{0}, {\bmat{D}}], \quad
    \tilde{{\bmat{M}}} = [\mathbf{0}, {\bmat{M}}]. 
\end{equation}
Furthermore, the discrete form of the weak formulation $\mathcal{W}_{\lambda} $ can be expressed as
\begin{eqnarray}
	\label{mat3}
		\begin{aligned}
			\delta \mathcal{W}_{\lambda, h}
			&= \delta \boldsymbol{\lambda}^{\top} {\bmat{D}} \bvec{d}_\mathcal{S} - \delta \boldsymbol{\lambda}^{\top} {\bmat{M}} \bvec{d}_\mathcal{M}
			= \delta \boldsymbol{\lambda}^{\top} \tilde{{\bmat{M}}}  \bvec{d}_1 - \delta \boldsymbol{\lambda}^{\top} \tilde{{\bmat{M}}} \mathrm{d}_2
			= \delta \boldsymbol{\lambda}^{\top} {\bmat{G}} \bvec{d}.
		\end{aligned}
\end{eqnarray}

In summary, by combining equations~\eqref{mat1},~\eqref{mat2}, and~\eqref{mat3}, the discrete system can be written as follows:
\begin{eqnarray}
	\label{eqn: twobody-disc}
	\begin{aligned}
		{\bmat{K}} \bvec{d} + {\bmat{G}}^{\top} \boldsymbol{\lambda} -\bvec{f}_{\mathrm{ext}} &= 0, \\
		{\bmat{G}} \bvec{d} &= 0,
	\end{aligned}
\end{eqnarray}
which can be expressed compactly by matrix-vector form
\begin{equation}
	\label{eqn: total}
	\left[ \begin{array}{c c }
		\bmat{K} & \bmat{G}^{\top} \\
		\bmat{G} & \mathbf{0}
	\end{array} \right]
	\left[ \begin{array}{c}
		\bvec{d} \\
		\boldsymbol{\lambda}
	\end{array} \right]
	=
	\left[ \begin{array}{c}
		\bvec{f} \\
		\mathbf{0}
	\end{array} \right],
\end{equation}
or succinctly expressed as 
\begin{eqnarray}
\label{eqn:lin-system-Ax-b}    
  \bmat{\mathcal{A}} \bvec{x} = \bvec{b}.
\end{eqnarray}

\subsection{Representation of the linear system for multi-body contact}
\label{subsect:present-multi-contact}

Consider the contact computation involving $n$ bodies, numbered from 1 to $n$.
Let $m$ represent the total number of contact surfaces, denoted as $f_1, \cdots, f_m$.
For each contact surface $f_i$, let $(b_{\mathcal{S}}^i, b_{\mathcal{M}}^i)$ represent the corresponding contact bodies, 
where $b_{\mathcal{S}}^i, b_{\mathcal{M}}^i \in \{1, \cdots, n \}$ denote the slave body and master body, respectively.

In the context of multi-body tied contact problems, the linear system can still be represented by~\eqref{eqn: total}.
Specifically, for $n$-body contact case, the matrix $\bmat{K}$ 
and the right-hand side $\bvec{f}$ can be expressed in block form:
\begin{equation}
	\bmat{K}=
	\left[
	\begin{array}{c c c}
		\bmat{K}_1  &  		 & \\
		            & \ddots &  \\
		              &  	   & \bmat{K}_{n}
	\end{array}
	\right], \quad	
	\bvec{f}=
	\left[
	\begin{array}{c}
		\bvec{f}_1 \\
		\vdots   \\
		\bvec{f}_{n}
	\end{array}
	\right],
\end{equation}
each subblock $\bmat{K}_i$ and $\bvec{f}_i$ corresponds to the stiffness matrix and external force for body $i$. 
In addition, the matrix $\bmat{G}$ takes the form:
\begin{equation}
	\bmat{G} =
	\left[
	\begin{array}{c}
		\bmat{G}_1 \\
		\vdots \\
		\bmat{G}_{m}
	\end{array}
	\right],
\end{equation}
and $\bmat{G}_i$ is a $1 \times n$ block matrix given by
\begin{equation}
    \begin{matrix}
    	\bmat{G}_i = &  \Big [ \cdots  &  \tilde{\bmat{D}}_{i} & \cdots & -\tilde{\bmat{M}}_{i} &  \cdots \Big ], \\
    	& & \downarrow & & \downarrow & \\
    	& & b_{\mathcal{S}}^i  & & b_{\mathcal{M}}^i  &
    \end{matrix}  
\end{equation}
where the $b_{\mathcal{S}}^i$th subblock is $\tilde{\bmat{D}}_{i}$, the $b_{\mathcal{M}}^i$th subblock is $-\tilde{\bmat{M}}_{i}$, and the remaining positions are filled with zeros.

The specific assembly of the linear system for multi-body contact
problem can be described in Algorithm~\ref{algorithm: discrete}.
In the algorithm, first, a classical finite element method is performed on each contact body to assemble the stiffness matrix 
$\bmat{K}_i$, followed by the discretization of each contact surface using the mortar element method.
Finally, an overall matrix $\bmat{\mathcal{A}}$ is formed.

\begin{algorithm}[htbp]
	\caption{Assembly of the linear system for tied contact problem.}
	\label{algorithm: discrete}
	\hspace*{0.02in} { Input:}
        contact model, number of contact bodies $n$ and number of contact surfaces $m$. \\
	\hspace*{0.02in} { Output:}
        The linear system with matrix $\bmat{\mathcal{A}}$, and right-hand side $\bvec{b}$.
	\begin{algorithmic}[1]
		\For{$i = 1 : n$}
		\State Assemble the stiffness matrix $\bmat{K}_i$ of the $i$th contact body.
		\EndFor
		\For{$j = 1 : m$}
		\State Assemble the mortar matrices $\bmat{D}_j$ and $\bmat{M}_j$.
		\EndFor
		\State  Assemble the global matrix $\bmat{\mathcal{A}}$ by $\bmat{K}_i,\bmat{D}_j$ and $\bmat{M}_j$, assemble the global right-hand side $\bvec{b}$.
	\end{algorithmic}
\end{algorithm}

The linear equation~\eqref{eqn: total} is a saddle point system with a $2 \times 2$ block structure, where the upper left subblock $\bmat{K}$ is a symmetric positive definite or symmetric positive semidefinite matrix. 
Due to the introduction of Lagrange multipliers in contact computation, the coefficient matrix of the saddle point system is not only larger than $\bmat{K}$, and most importantly, the matrix is indefinite.
These factors significantly make it difficult to solve the system.
In the next section, an efficient solution method is proposed 
for this symmetric indefinite saddle point system.

\section{Condensation method for saddle point linear systems}
\label{sect:method}

In this section, 
an effective condensation elimination method is proposed to transform the linear saddle point system of two-dimensional tied contact problems into a symmetric positive definite system.
The method is based on condensing Lagrange multipliers by the structure of the mortar matrices associated with the slave surface. 
Therefore, the method can not only reduce the size of the original system but also transform a symmetric positive indefinite 
linear system into a symmetric positive definite system.  
This provides significant convenience for solving linear systems in contact problems.

In addition, the elimination process for multiple contact surfaces is independent of each other, 
allowing for parallel implementation.

\subsection{Elimination by condensation} 
\label{subsect:elim-conden}

For ease of introducing the elimination algorithm, 
all nodes are segregated into three subsets: the master node set $\mathcal{M}$, the slave node set $\mathcal{S}$, and the remaining node set $\mathcal{N}$.
The unknowns are accordingly partitioned as $\bvec{x} = [\bvec{d}_{\mathcal{N}},\bvec{d}_\mathcal{M},\bvec{d}_\mathcal{S},\boldsymbol{\lambda}]$, and the linear system~\eqref{eqn: total} can be equivalently written as:
\begin{equation}
	\label{twobody1}
	\begin{small}
		\left[ \begin{array}{c c c c}
			\bmat{K}_{\mathcal{N} \mathcal{N}} & \bmat{K}_{\mathcal{N} \mathcal{M}} & \bmat{K}_{\mathcal{N} \mathcal{S}}  & \mathbf{0} \\
			\bmat{K}_{\mathcal{M} \mathcal{N}} & \bmat{K}_{\mathcal{M} \mathcal{M}} & \mathbf{0} 							 & -\bmat{M}^{\top} \\
			\bmat{K}_{\mathcal{S} \mathcal{N}} & \mathbf{0} 						  & \bmat{K}_{\mathcal{S} \mathcal{S}}  & \bmat{D}^{\top}  \\
			\mathbf{0} 							& -\bmat{M} 						  & \bmat{D} 							 & \mathbf{0}
		\end{array} \right]
		\left[ \begin{array}{c}
			\bvec{d}_{\mathcal{N}} \\ 
			\bvec{d}_{\mathcal{M}} \\ 
			\bvec{d}_{\mathcal{S}} \\ 
			\boldsymbol{\lambda}
		\end{array} \right]
		=
		\left[ \begin{array}{c}
			\bvec{f}_{\mathcal{N}} \\
			\bvec{f}_{\mathcal{M}} \\
			\bvec{f}_{\mathcal{S}}\\
			\mathbf{0}
		\end{array} \right].
	\end{small}
\end{equation}

The elimination process can be performed in two steps.
The first step is to eliminate the Lagrange multiplier degrees of freedom, followed by the elimination of the slave side displacement degrees of freedom in the second step.
Firstly, by the third row of equation~\eqref{twobody1}, the unknown $\boldsymbol{\lambda}$ can be expressed as follows:
\begin{equation}
	\label{eqn: lambda-exp}
	\boldsymbol{\lambda} =
	\bmat{D}^{-\top} ( \bvec{f}_{\mathcal{S}}  - 
	\bmat{K}_{\mathcal{S} \mathcal{N}} \bvec{d}_{\mathcal{N}}
	-
	\bmat{K}_{\mathcal{S} \mathcal{S}} \bvec{d}_{\mathcal{S}}
	).
\end{equation}
By substituting equation~\eqref{eqn: lambda-exp} into the second row of equation~\eqref{twobody1} and eliminating the unknown $\boldsymbol{\lambda}$, 
we can obtain the following system:
\begin{equation}
	\label{step1}
	\left[ \begin{array}{c c c}
		\bmat{K}_{\mathcal{N} \mathcal{N}} 								  & \bmat{K}_{\mathcal{N} \mathcal{M}} & \bmat{K}_{\mathcal{N} \mathcal{S}} \\
		\bmat{K}_{\mathcal{M} \mathcal{N}} + \bmat{P}^{\top} \bmat{K}_{\mathcal{S} \mathcal{N}} & \bmat{K}_{\mathcal{M} \mathcal{M}} & \bmat{P}^{\top} \bmat{K}_{\mathcal{S} \mathcal{S}} \\
			\mathbf{0} 															  & -\bmat{M} 						    & \bmat{D} 
	\end{array} \right]
	\left[ \begin{array}{c}
		\bvec{d}_{\mathcal{N}} \\ 
		\bvec{d}_{\mathcal{M}} \\ 
		\bvec{d}_{\mathcal{S}} 
	\end{array} \right]
	= 
	\left[ \begin{array}{c}
		\bvec{f}_{\mathcal{N}} \\
		\bvec{f}_{\mathcal{M}} + \bmat{P}^{\top} \bvec{f}_{\mathcal{S}} \\
		\mathbf{0}
	\end{array} \right],
\end{equation}
where $\bmat{P}=\bmat{D}^{-1}\bmat{M}$. 
Secondly, the last row of equation~\eqref{step1} yields:
\begin{equation}
	\label{eqn: ds-exp}
	\bvec{d}_{\mathcal{S}} = \bmat{D}^{-1} \bmat{M} \bvec{d}_{\mathcal{M}} = \bmat{P} \bvec{d}_{\mathcal{M}},
\end{equation}
that is, the slave side displacement $\bvec{d}_{\mathcal{S}}$ can be expressed in terms of  the master side displacement $\bvec{d}_{\mathcal{M}}$.
Substituting equation~\eqref{eqn: ds-exp} into the first two lines of equation~\eqref{step1} and eliminating $\bvec{d}_{\mathcal{S}}$ results in:
\begin{equation}
	\label{twobody2}
	\left[ \begin{array}{c c}
		\bmat{K}_{\mathcal{N} \mathcal{N}} & \bmat{K}_{\mathcal{N} \mathcal{M}} + \bmat{K}_{\mathcal{N} \mathcal{S}} \bmat{P} \\
		\bmat{K}_{\mathcal{M} \mathcal{N}} + \bmat{P}^{\top} \bmat{K}_{\mathcal{S} \mathcal{N}} & \bmat{K}_{\mathcal{M} \mathcal{M}} + \bmat{P}^{\top} \bmat{K}_{\mathcal{S} \mathcal{S}} \bmat{P}
	\end{array} \right]
	\left[ \begin{array}{c}
		\bvec{d}_{\mathcal{N}} \\ 
		\bvec{d}_{\mathcal{M}} 
	\end{array} \right]
	=
	\left[ \begin{array}{c}
		\bvec{f}_{\mathcal{N}} \\
		\bvec{f}_{\mathcal{M}} + \bmat{P}^{\top} \bvec{f}_{\mathcal{S}}
	\end{array} \right].
\end{equation}

The derivation above shows that system~\eqref{twobody1} is equivalent to the system~\eqref{twobody2}. This is achieved by eliminating the Lagrange multipliers and slave side displacements, reducing the overall size of the system.
The coefficient matrix after elimination is symmetric positive definite (Proposition~\ref{prop:hat-A-spd}). 
Thus the preconditioned conjugate gradient (PCG) method can be 
used to solve the system.

While some degrees of freedom are condensed and the size of
the linear system is reduced, some fillings are introduced for
the matrix of the reduced system. 
By equation~\eqref{twobody2}, only in the rows corresponding to the master side degrees of freedom $\bvec{d}_{\mathcal{M}}$ and the nonzero rows of $\bmat{K}_{\mathcal{N} \mathcal{S}}$,  some more 
filling will be introduced.
However, the percentage of rows with further fillings
is small compared to the overall equation.

To perform the aforementioned elimination process, it is 
essential to compute the matrix 
$\bmat{P} = \bmat{D}^{-1} \bmat{M}$. 
To address the contact problems of multiple elastic bodies,
we need to divide the degrees of freedom on slave nodes and master nodes as
\begin{equation}
    \bvec{d}_{\mathcal{S}} =  [\bvec{d}_{\mathcal{S},1}, \cdots, \bvec{d}_{\mathcal{S},m}],
    \quad
    \bvec{d}_{\mathcal{M}} =  [\bvec{d}_{\mathcal{M},1}, \cdots, \bvec{d}_{\mathcal{M},m}],
\end{equation}
where $m$ denotes the number of contact surfaces. 
Thus, the global mortar matrix $\bmat{D}$ and $\bmat{M}$ can be formulated as follows:
\begin{equation}
		\bmat{D}=
		\left[ \begin{array}{c c c}
			\bmat{D}_1 &  		 & 			\\
			& \ddots &  		\\
			&  		 & \bmat{D}_{m}
		\end{array} \right], \quad
		\bmat{M}=
		\left[ \begin{array}{c c c}
			\bmat{M}_1 &  		 & 		 	\\
			& \ddots &  		\\
			&  		 & \bmat{M}_{m}
		\end{array} \right]	.
\end{equation}
Accordingly, matrix $\bmat{P}$ can be formulated as:
\begin{equation}\label{P}
		\begin{aligned}
			\bmat{P} = \bmat{D}^{-1} \bmat{M}
			=
			\left[ \begin{array}{c c c}
				\bmat{D}_1^{-1} \bmat{M}_1 &  	  & 		\\
				& \ddots &  		\\
				&  	  & \bmat{D}_{m}^{-1} \bmat{M}_{m}
			\end{array} \right]	
			=
			\left[ \begin{array}{c c c}
				\bmat{P}_1 &  		 &    \\
				& \ddots &    \\
				&  		 & \bmat{P}_{m}
			\end{array} \right],
		\end{aligned}
\end{equation}
and $\bmat{P}$ is a block diagonal matrix, where the number of blocks is equal to the number of contact surfaces.
Equation~\eqref{P} shows that the computation of $\bmat{P}_i = \bmat{D}^{-1}_i \bmat{M}_i$ is independent at each contact surface. 
If the mesh of the contact surface is matching, 
then $\bmat{D}_i = \bmat{M}_i$, and $\bmat{P}_i$ is an identity matrix. 
In the case that the mesh is non-matching, it is necessary to compute the subblocks of the matrix $\bmat{P}$, which satisfies the equation $\bmat{P} = \bmat{D}^{-1} \bmat{M}$, or equivalently 
\begin{eqnarray}
	\label{eqn: mat-equ}
	\bmat{D} \bmat{P} = \bmat{M}.
\end{eqnarray}
To obtain the matrix $\bmat{P}$, the above matrix equation needs to be solved.

\subsection{Solving matrix equations based on block Thomas algorithm}
\label{subsect:block-Thomas}

The implementation of the condensation process concerns
the solution for the matrix equation~\eqref{eqn: mat-equ}. 
For the two-dimensional tied contact problem, 
by exploring the application characteristics and structural features of the mortar matrix $\bmat{D}$, 
a block Thomas algorithm is used to solve the equation.
For clarity of description of the method, 
an example with a single contact surface is provided below.

Assuming that only one contact surface exists between two contact bodies and it contains $n$ contact nodes.
The elements of the mortar matrix $\bmat{D}$   
can be represented as:
\begin{equation}
\begin{aligned}
   \bmat{D}[j, k] &= D_{j, k} \bmat{I}_{\textsf{ndim }}
   = \left( \int_{\Gamma_{\mathrm{c}, h}^{(1)}} \Phi_{j} N_{k}^{(1)} \mathrm{d} S \right) \bmat{I}_{\textsf{ndim}}, \quad 
   j, k = 1, \ldots, n,
\end{aligned}
\end{equation}
where Lagrange multiplier basis function $\Phi_j$ and slave side displacement basis function $N_{k}^{(1)}$ are both linear functions.
The matrix $\bmat{D}$ contains subblocks $\bmat{D}[j,k]$ corresponding to the $j$-th and $k$-th slave nodes, 
where each subblock is a $\textsf{ndim} \times \textsf{ndim}$ matrix.
By the property of the linear finite element basis function,
it is easy to know that
$\bmat{D}[j, k] \not= 0$ if and only if $j$ and $k$ represent the same or adjacent spatial node, and $\bmat{D}[j, k] = 0$ otherwise.
If the nodes are numbered by their spatial adjacency relationship, the $j$-th row block of matrix $\bmat{D}$ (representing the row block of matrix $\bmat{D}$ at node $j$) contains only three non-zero components: $\bmat{D}[j,j-1], \bmat{D}[j,j]$, and $\bmat{D}[j,j+1]$. 
In short, the matrix $\bmat{D}$ exhibits a tridiagonal block structure.
For the sake of convenience, in the following, we will use the notation $\bmat{D}_{j,k}$ to represent the subblock matrix $\bmat{D}[j, k]$.
Here's the specific form of $\bmat{D}$:
\begin{equation}
	\begin{small}
		\bmat{D} =
		\left[ \begin{array}{c c c c c}
			\bmat{D}_{1,1} & \bmat{D}_{1,2} &             &	          &		 	   \\
			\bmat{D}_{2,1} & \bmat{D}_{2,2} & \bmat{D}_{2,3}     &             &		  	   \\
			& \ddots  &     \ddots 	& \ddots  	  &			   \\
			&	      & \bmat{D}_{n-1,n-2} & \bmat{D}_{n-1,n-1} & \bmat{D}_{n-1,n}  \\
			&		  & 	        & \bmat{D}_{n,n-1}	  & \bmat{D}_{n,n}
		\end{array} \right], \quad 	
	\end{small}
\end{equation}
where
\begin{equation*}
    \bmat{D}_{j,k} =
    \left[ \begin{array}{c c}
    	D_{j, k} &                 \\
    	        & D_{j, k}
    \end{array}	\right] .
\end{equation*}

By the structure of matrix $\bmat{D}$, 
we can perform an LU decomposition 
$\bmat{D} = \bmat{L}\bmat{U}$, where
\begin{equation}
	\begin{aligned}
		\bmat{L} =
		\left[ \begin{array}{c c c c c}
			\bmat{I} 	&    	 &         &	     &		          			\\
			\bmat{L}_2 & \bmat{I} 	 &         &      	 & 		   	     			\\
			& \ddots & \ddots  &  	     &   	  	     			\\
			&  		 & \bmat{L}_{n-1} &	\bmat{I}    &		  		 			\\
			&  		 &         & \bmat{L}_{n}   &	\bmat{I}	  			 		\\	
		\end{array} \right], \quad 
		\bmat{U} =
		\left[ \begin{array}{c c c c c}
			\bmat{U}_1 & \bmat{D}_{1,2} &         &	      &		      				\\
			& \bmat{U}_2     & \bmat{D}_{2,3} &    	  & 	   					\\
			& 	  	  & \ddots  & \ddots  &   	     			 	\\
			&  	  	  &         & \bmat{U}_{n-1} & \bmat{D}_{n-1,n}				\\
			&  	  	  &         &  		  & \bmat{U}_{n}		 			\\	
		\end{array} \right]	,
	\end{aligned}
\end{equation}
each subblock is a $2 \times 2$ matrix:
\begin{equation}
	\bmat{L}_i =
	\left[ \begin{array}{c c}
		l_{i} &                 \\
		& l_{i}
	\end{array}	\right], \;
	\bmat{I} =
	\left[ \begin{array}{c c}
		1 &                 \\
		& 1
	\end{array}	\right], \;
	\bmat{U}_i =
	\left[ \begin{array}{c c}
		u_{i} &                 \\
		& u_{i}
	\end{array}	\right] .
\end{equation}
The specific computation for $l_i$  and $u_i$ 
is as follows:
\begin{equation}
	\begin{aligned}
		&u_i = D_{i,i}  \, &i &= 1 \\
		&l_i = D_{i,i-1} / u_{i-1} , \quad  u_i = D_{i,i} - l_i  D_{i-1,i}  \,  &i &= 2, \cdots, n
	\end{aligned}
\end{equation}
Thus, the matrix equation $\bmat{D} \bmat{P}=\bmat{M}$ becomes
$\bmat{L} \bmat{U} \bmat{P} = \bmat{M}$.
To solve this matrix equation, the ``forward'' and ``backward'' processes can be employed as follows:
\begin{enumerate}
	\item[(a)] Solve the block lower triangular equation $\bmat{L} \bmat{Y} = \bmat{M}$ to obtain $\bmat{Y}$.
	\item[(b)] Solve the block upper triangular equation $\bmat{U} \bmat{P} = \bmat{Y}$ to get $\bmat{P}$.
\end{enumerate}
Specifically, the two matrix equations can be solved by sequentially updating each row block:
\begin{enumerate}
	\item[(a)] Compute the row block of matrix $\bmat{Y}$ by
	\begin{eqnarray}
		\label{eqn: low-tri-equ-sol}
			\left\{
			\begin{aligned}
				\bmat{Y}(i, :) &= \bmat{M}(i, :), \, &i &= 1 \\
				\bmat{Y}(i, :) &= \bmat{M}(i, :) - \bmat{L}_i \bmat{Y}(i-1, :), \, &i &= 2, \cdots, n
			\end{aligned}
			\right.
	\end{eqnarray}
	\item[(b)] Compute the row block of matrix $\bmat{P}$ by
	\begin{eqnarray}
		\label{eqn: up-tri-equ-sol}
			\left\{
			\begin{aligned}
				\bmat{P}(i, :) &= \bmat{Y}(i, :)  \bmat{U}_i^{-1}, \, &i &= n \\
				\bmat{P}(i, :) &= \left(\bmat{Y}(i, :) - \bmat{D}_{i,i+1} \bmat{P}(i+1, :) \right)  \bmat{U}_i^{-1} , \, &i &= n-1, \cdots, 1
			\end{aligned}
			\right.
	\end{eqnarray}
\end{enumerate}
In~\eqref{eqn: low-tri-equ-sol} and~\eqref{eqn: up-tri-equ-sol}, $\bmat{Y}(i,:)$ denotes the $i$-th row block of matrix $\bmat{Y}$, which contains the rows $2i - 1$ and $2i$. 
The same notation applies to matrices $\bmat{M}$ and $\bmat{P}$.

For multiple contact surface problems, we can solve the corresponding matrix equation $\bmat{D}_i \bmat{P}_i= \bmat{M}_i$ separately on each surface. 
The computation for different contact surfaces is independent, allowing for parallel computation on multiple surfaces.

It is important to note that, by the aforementioned elimination process, the matrix equation $\bmat{D} \bmat{P} = \bmat{M}$ 
is solved exactly. 
The above process is a generalization of Thomas algorithm for solving linear equations~\cite{press2007numerical}.
The computation only involves vector correction operations, making it relatively simple to implement.

\subsection{Matrix representation}
\label{subsect:matrix-present}

In Subsection~\ref{subsect:elim-conden}, a detailed description of the specific process of the condensed elimination algorithm is presented. 
The elimination process will be described in matrix form in this section for the purpose of describing the algorithm concisely.

Define a matrix $\bmat{T}$ by
\begin{equation}\label{eqn: T}
	\bmat{T}
	=
	\left[ \begin{array}{c c c c}
		\bmat{I}_{\mathcal{S} \mathcal{S}} & 							  &   	 					    & 		\\
		& \bmat{I}_{\mathcal{N} \mathcal{N}} &  						    &		\\
		\bmat{P}^{\top}						&  	  						  & \bmat{I}_{\mathcal{M} \mathcal{M}} &       \\
		& 							  & 							& \bmat{I}_{\boldsymbol{\lambda} \boldsymbol{\lambda}}
	\end{array} \right]	,
\end{equation}
where the matrix $\bmat{P}$ is given by~\eqref{P}.
By the matrix $\bmat{T}$, equation~\eqref{eqn:lin-system-Ax-b} can be transformed as
\begin{equation}
	\bmat{T} \boldsymbol{\mathcal{A}} \bmat{T}^{\top} (\bmat{T}^{\top})^{-1} \bvec{x} = \bmat{T} \bvec{b} ,
\end{equation}
or specifically,
\begin{equation}\label{TATt}
    \left[ \begin{array}{c c c c}   
        \bmat{K}_{\mathcal{S} \mathcal{S}}	& \bmat{K}_{\mathcal{S} \mathcal{N}} & \bmat{K}_{\mathcal{S} \mathcal{S}} \bmat{P} & \bmat{D}^{\top} \\
	\bmat{K}_{\mathcal{N} \mathcal{S}} & \bmat{K}_{\mathcal{N} \mathcal{N}} & \bmat{K}_{\mathcal{N} \mathcal{M}} + \bmat{K}_{\mathcal{N} \mathcal{S}} \bmat{P} & \mathbf{0} \\
	\bmat{P}^{\top} \bmat{K}_{\mathcal{S} \mathcal{S}} & \bmat{K}_{\mathcal{M} \mathcal{N}} + \bmat{P}^{\top} \bmat{K}_{\mathcal{S} \mathcal{N}} & \bmat{K}_{\mathcal{M} \mathcal{M}} + \bmat{P}^{\top} \bmat{K}_{\mathcal{S} \mathcal{S}} \bmat{P} & \mathbf{0} \\
	\bmat{D} & \mathbf{0} & \mathbf{0} & \mathbf{0}
    \end{array} \right]
    \left[ \begin{array}{c}
	\bvec{d}_{\mathcal{S}} - \bmat{P} \bvec{d}_{\mathcal{M}} \\
	\bvec{d}_{\mathcal{N}} \\
	\bvec{d}_{\mathcal{M}} \\
	\boldsymbol{\lambda}
    \end{array} \right]
    =
    \left[ \begin{array}{c}
	\bvec{f}_{\mathcal{S}} \\
	\bvec{f}_{\mathcal{N}} \\
	\bvec{f}_{\mathcal{M}} + \bmat{P}^{\top} \bvec{f}_{\mathcal{S}} \\
	\mathbf{0}
    \end{array} \right].
\end{equation}
Next introduce a matrix
\begin{equation}\label{eqn: C}
	\bmat{C}
	=
	\left[ \begin{array}{c c c c}
		\mathbf{0} & \bmat{I}_{\mathcal{N} \mathcal{N}} & \mathbf{0} 						  & \mathbf{0}		\\
		\mathbf{0} & \mathbf{0} 							& \bmat{I}_{\mathcal{M} \mathcal{M}} & \mathbf{0}
	\end{array} \right].
\end{equation}
By the matrix $\bmat{C}$, 
the elimination of the degrees of freedom can be achieved by removing certain rows and columns of the matrix in equation~\eqref{TATt}. 
In fact, it is easy to check that the matrix 
$\bmat{C}\bmat{T} \boldsymbol{\mathcal{A}} \bmat{T}^{\top} \bmat{C}^{\top}$ is the coefficient matrix of the system after elimination, as shown in equation~\eqref{twobody2}.

In summary, the process of condensed elimination 
(DOFs condensation based method) can be succinctly described as follows:
\begin{enumerate}
	\item[(1)] Define the matrix $\bmat{T}$ as shown in equation~\eqref{eqn: T};
	\item[(2)] Define the matrix $\bmat{C}$ as shown in equation~\eqref{eqn: C};
	\item[(3)] Compute the matrix $\bmat{F} = \bmat{C}\bmat{T}$;
	\item[(4)] Solve the linear system
	\begin{eqnarray}
		\label{eqn: FAFt}
		\hat{\boldsymbol{\mathcal{A}}} \hat{\bvec{x}} = \hat{\bvec{b}},
		\quad
		\text{where}
		\quad
		\hat{\boldsymbol{\mathcal{A}}} = \bmat{F} \boldsymbol{\mathcal{A}} \bmat{F}^{\top},
		\quad
		\hat{\bvec{b}} = \bmat{F} \bvec{b}.
	\end{eqnarray}
\end{enumerate}
Here, $\hat{\boldsymbol{\mathcal{A}}}$, $\hat{\bvec{x}}$, and $\hat{\bvec{b}}$ represent the condensed stiffness matrix, DOFs vector, and right-hand side vector after the elimination, respectively.
In fact, $\hat{\bvec{x}}$ is given as 
\begin{equation}
	\hat{\bvec{x}} = [\bvec{d}_\mathcal{N}, \bvec{d}_\mathcal{M}]^{\top}.
\end{equation}
The matrix $\bmat{F}$ is referred to as the {\em elimination matrix}, which is the product of the condensed matrix $\bmat{C}$ and the elementary row transformation matrix $\bmat{T}$. 

In the following, we will devote to the analysis of 
the property of the matrix $\hat{\boldsymbol{\mathcal{A}}}$.

\begin{proposition} \label{proposition1}
    If the original matrix $\boldsymbol{\mathcal{A}}$ is nonsingular, then the condensed matrix $\hat{\boldsymbol{\mathcal{A}}}$ is also nonsingular.
\end{proposition}

\begin{proof}
    Assume that the matrix $\hat{\boldsymbol{\mathcal{A}}}$ is singular, then the equation~\eqref{eqn: FAFt} has at least two distinct solutions. 
    Let two vectors $\bvec{u}_1$ and $\bvec{u}_2$, 
    $\bvec{u}_1 \neq \bvec{u}_2$, such that 
    $\hat{\boldsymbol{\mathcal{A}}} \bvec{u}_1 = \hat{\boldsymbol{\mathcal{A}}} \bvec{u}_2 =  \hat{\bvec{b}} $.     
    Let 
    \begin{equation}
        \bvec{u}_1 = 
        \left[ 
        \bvec{d}_{\mathcal{N}}^{(1)}, \bvec{d}_{\mathcal{M}}^{(1)} 
        \right]^{\top},
        \quad
        \bvec{u}_2 =
        \left[
        \bvec{d}_{\mathcal{N}}^{(2)}, \bvec{d}_{\mathcal{M}}^{(2)}
        \right]^{\top}.
    \end{equation}
    By~\eqref{eqn: lambda-exp} and~\eqref{eqn: ds-exp},
    the eliminated degrees of freedom $\boldsymbol{\lambda}$ and $\bvec{d}_{\mathcal{S}}$ can be recovered by the remaining degrees of freedom, 
    thus we can construct two vectors 
    $\bvec{v}_1$ and $\bvec{v}_2$ by
    \begin{eqnarray*}
    \bvec{v}_1 = 
    \left[ \bvec{d}_{\mathcal{S}}^{(1)}, \bvec{d}_{\mathcal{N}}^{(1)}, \bvec{d}_{\mathcal{M}}^{(1)}, \boldsymbol{\lambda}^{(1)} 
    \right]^{\top} 
    = 
    \left[ \bmat{P} \bvec{d}_{\mathcal{M}}^{(1)}, \bvec{d}_{\mathcal{N}}^{(1)}, \bvec{d}_{\mathcal{M}}^{(1)},
	\bmat{D}^{-{\top}} \left( \bvec{f}_{\mathcal{S}} - 	\bmat{K}_{\mathcal{S} \mathcal{N}} \bvec{d}_{\mathcal{N}}^{(1)} - 
	\bmat{K}_{\mathcal{S} \mathcal{S}} \bvec{d}_{\mathcal{S}}^{(1)} \right) 
    \right]^{\top}, \\
    \bvec{v}_2 = 
    \left[ \bvec{d}_{\mathcal{S}}^{(2)}, \bvec{d}_{\mathcal{N}}^{(2)}, \bvec{d}_{\mathcal{M}}^{(2)}, \boldsymbol{\lambda}^{(2)} 
    \right]^{\top} 
    = 
    \left[ \bmat{P} \bvec{d}_{\mathcal{M}}^{(2)}, \bvec{d}_{\mathcal{N}}^{(2)}, \bvec{d}_{\mathcal{M}}^{(2)},
	\bmat{D}^{-{\top}} \left( \bvec{f}_{\mathcal{S}} - 	\bmat{K}_{\mathcal{S} \mathcal{N}} \bvec{d}_{\mathcal{N}}^{(2)} - 
	\bmat{K}_{\mathcal{S} \mathcal{S}} \bvec{d}_{\mathcal{S}}^{(2)} \right) 
    \right]^{\top}.     
    \end{eqnarray*}
    Both $\bvec{v}_1$ and $\bvec{v}_2$ are solutions of 
    the original equation~\eqref{eqn: total}, 
    that is, $ \boldsymbol{\mathcal{A}} \bvec{v}_1 = \boldsymbol{\mathcal{A}} \bvec{v}_2 = \bvec{b} $.
    Since $ \bvec{u}_1 \neq \bvec{u}_2 $, it is evident that $\bvec{v}_1 \neq \bvec{v}_2$. 
    This contradicts the assumption that $ \boldsymbol{\mathcal{A}} $ is nonsingular.
    Therefore, the matrix $\hat{\boldsymbol{\mathcal{A}}}$ is nonsingular. 
\end{proof}

\begin{proposition}
\label{prop:hat-A-spd}
    The condensed matrix $\hat{\boldsymbol{\mathcal{A}}}$ is symmetric positive definite.
\end{proposition}

\begin{proof}
    The matrix of contact problems~\eqref{eqn: total} is:
    \begin{equation*}
        \boldsymbol{\mathcal{A}} = 
        \left[ \begin{array}{c c }
            \bmat{K} & \bmat{G}^{\top} \\
            \bmat{G} & \mathbf{0}
        \end{array} \right],        
    \end{equation*}
    assume that there are $n$ contact bodies, then the matrix $\bmat{K}$ has the form $\bmat{K} = \mathrm{diag(\bmat{K}_1, \dots, \bmat{K}_n)}$, where $\bmat{K}_i$ is the elsticity disrete operator for body $i$. Since each $\bmat{K}_i$ is symmetric positive
    definite or symmetric positive semi-definite, the matrix $\bmat{K}$ is 
    symmetric positive semi-definite. 

    Let
    \begin{eqnarray*}
    \tilde{\bmat{T}} =  
	\left[ \begin{array}{ccc}
		\bmat{I}_{\mathcal{S} \mathcal{S}}   &  & \\
		& \bmat{I}_{\mathcal{N} \mathcal{N}} &    \\
		\bmat{P}^{\top}						 &  & \bmat{I}_{\mathcal{M} \mathcal{M}}        
	\end{array} \right],    
    \quad
    \tilde{\bmat{C}} = 
    \left[ \begin{array}{ccc}
		\mathbf{0} & \bmat{I}_{\mathcal{N} \mathcal{N}} & \mathbf{0} 			\\
		\mathbf{0} & \mathbf{0} 							& \bmat{I}_{\mathcal{M} \mathcal{M}} 
	\end{array} \right],
    \end{eqnarray*}
    where $\tilde{\bmat{T}}$ is obtained by removing the last row and last column blocks from matrix $\bmat{T}$, 
    and $\tilde{\bmat{C}}$ is constructed by removing the last column blocks from matrix $\bmat{C}$.     
    Then by~\eqref{eqn: FAFt}, it is easy to verify that the following hold
     \begin{equation}
	\hat{\boldsymbol{\mathcal{A}}} = \tilde{\bmat{C}} \tilde{\bmat{T}} \bmat{K} \tilde{\bmat{T}}^{\top} \tilde{\bmat{C}}^{\top}.
    \end{equation}
    Thus, the matrix $\hat{\boldsymbol{\mathcal{A}}}$ is 
    symmetric positive semi-definite.
    Furthermore, by~\ref{proposition1}, matrix $\hat{\boldsymbol{\mathcal{A}}}$ is nonsingular,
    Thus, $\hat{\boldsymbol{\mathcal{A}}}$ is a symmetric positive definite matrix.    

\end{proof}

The original matrix $\boldsymbol{\mathcal{A}}$ before elimination has a saddle point structure. 
After DOFs condensation, the matrix $\hat{\boldsymbol{\mathcal{A}}}$ becomes symmetric positive definite.
This makes the linear system easier to be solved.

\subsection{Condensation algorithm}

The degree of freedom condensation method, 
as discussed in Subsection~\ref{subsect:matrix-present}, involves the construction of 
the matrix $\bmat{T}$ and $\bmat{C}$. Furthermore, 
the construction of the matrix $\bmat{T}$ concerns the construction
of the matrix $\bmat{P} = \bmat{D}^{-1} \bmat{M}$.
In Subsection~\ref{subsect:block-Thomas}, 
a detailed explanation has been provided 
about the computation of $\bmat{P}$.

The structure of the matrix $\bmat{D}$ is dependent on the global numbering of surface nodes. 
Only if the numbering of nodes on the contact surface is spatially adjacent, can the mortar matrix $\bmat{D}$ have a tridiagonal block structure.
If the numbering of nodes on a contact surface 
is not spatically adjacent, it is needed to resort the node numbering. 






As an illustration, a simple model is shown in Fig.~\ref{slave_num} where the numbering is adjacent on the contact surface.
Fig.~\ref{D5} presents the associated structure of the mortar matrix $\bmat{D}$. After reordering the node number based
on the space adjacent, 
the corresponding mortar matrix $\bmat{D}$ has 
a tridiagonal block structure as shown in Fig.~\ref{D52}.


\begin{figure}[htbp]
	\centering
	\includegraphics[scale=0.28]{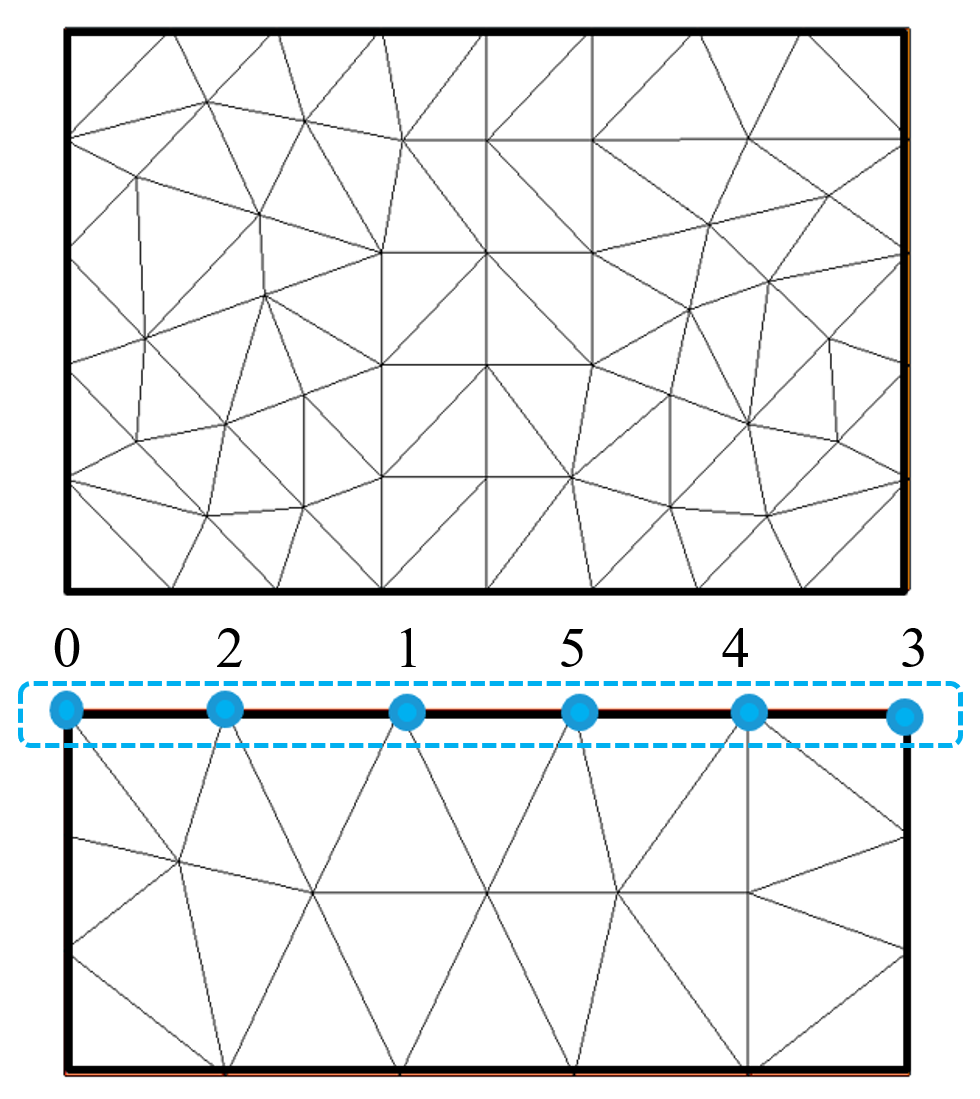}
	\caption{Node number of slave contact surface.}\label{slave_num}
\end{figure}

\begin{figure}[htbp]
	\centering
	\subfigure[Matrix pattern before rearrangement]{\includegraphics[scale=0.3]{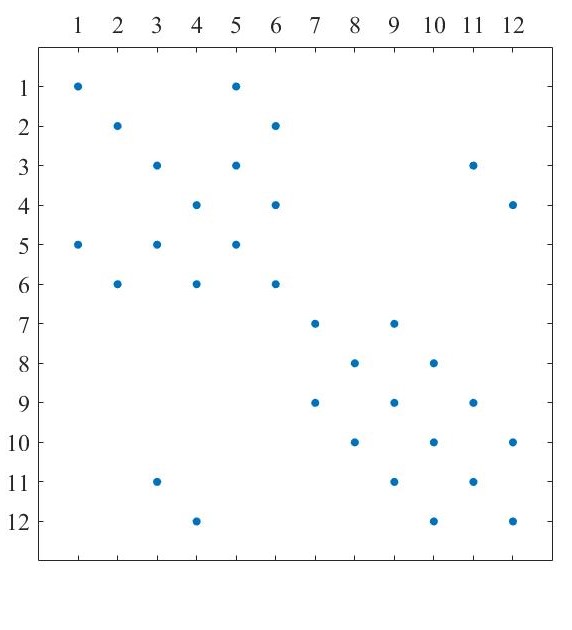} \label{D5}}
	\hspace{1in}
	\subfigure[Matrix pattern after rearrangement]{\includegraphics[scale=0.3]{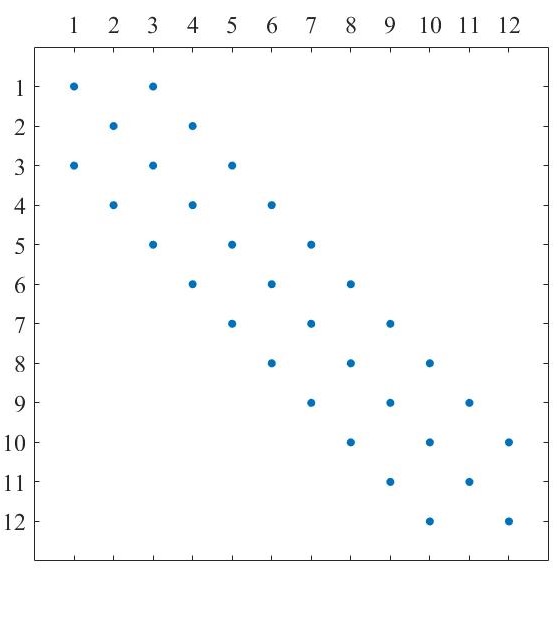} \label{D52}}
	\caption{ Matrix pattern of slave mortar matrix $\bmat{D}$.}
	\label{aa}
\end{figure}

The DOFs condensation based algorithm can now be described in Algorithm~\ref{algorithm: solve}.
In this algorithm, 
the matrix $\bmat{T}$ is constructed from Line 1--5.
Line 2--4 can be implemented in parallel because 
the computation of $\bmat{P}_i$ is independent 
for all the contact surfaces.


\begin{algorithm}
    \caption{DOFs condensation based algorithm for tied contact problem.} 
    \label{algorithm: solve}
    \hspace*{0.02in} { Input:}
    contact matrix $\boldsymbol{\mathcal{A}}$, and right-hand side $\bvec{b}$. \\
    \hspace*{0.02in} { Output:} 
    Solution of contact problems
    \begin{algorithmic}[1]
	\For{$j \in \{1, 2, \ldots, m\}$}
		\State Resort the slave surface nodes if necessary so that the node numbering is contiguous for adjacent nodes.
		\State Solve matrix equations $ \bmat{D}_j \bmat{P}_j = \bmat{M}_j $ by
  ~\eqref{eqn: low-tri-equ-sol} and~\eqref{eqn: up-tri-equ-sol} (block Thomas algorithm).  
		\State Assemble $\bmat{P}_j$ into matrix $\bmat{T}$.
	\EndFor
	\State Construct the condensation matrix $\bmat{C}$ by the global numbering of the slave nodes.
	\State Compute $\bmat{F} = \bmat{C} \bmat{T}$, 	 
 $\hat{\boldsymbol{\mathcal{A}}} = \bmat{F} \boldsymbol{\mathcal{A}} \bmat{F}^{\top}$, and $\hat{\bvec{b}} = \bmat{F} \bvec{b}$.
	\State Solve the linear system 
               $\hat{\boldsymbol{\mathcal{A}}}  \hat{\bvec{x}} = \hat{\bvec{b}}$. \\
	\Return solution $\hat{\bvec{x}}$.
    \end{algorithmic}
\end{algorithm}



\section{Numerical Results}
\label{sect:numer}

In this section, numerical results of the DOFs condensation method
is given for three contact models. 


The numerical experiments are conducted on a cluster, 
which has 28 cores and 96G memory on each node.

\subsection{Test models}
\label{subsect:model}

Three tied contact two-dimensional models are tested.  
The specific test models are given in Fig.~\ref{model},
where the first two models have three contact bodies,
and the last one has two contact bodies.
All the test models are linear elastic, 
and the material parameters include Young's modulus $E$ and Poisson's ratio $\nu$. 
In the tests, we set $E = 20$ N/m$^2$ and $\nu = 0.3$.

For the first and second test models, 
the middle body is the master body, and the rest are slave bodies, and this corresponds to two contact surfaces.
For the third model, the body on the top is 
the master body and on the bottom is the slave body,
and one contact surface is set.

For the first model, the boundary on the left is fixed 
and on the right boundary, 
a constant surface force of $p=10 \text{N}$ in the rightward direction is set.
For the second model, the lower boundaries of the three bodies are fixed, 
and an uniform surface force of $p=10 \text{N}$ in the downward direction
is set on the upper boundaries. 
For the third model, the lower boundary of the slave body is fixed, 
and an uniformly distributed surface force of $p=-1 \text{N}$ in the downward direction
is set on the upper boundary of the master body. 
For all three models, no body force is set.

All the test models are partitioned into triangular meshes, 
and the package FreeFEM++~\cite{FREEFEMhecht2012new} is employed to discretize the tied contact problem.

\begin{figure}[H]
	\centering
	\subfigure[Model 1]{\includegraphics[scale=0.3]{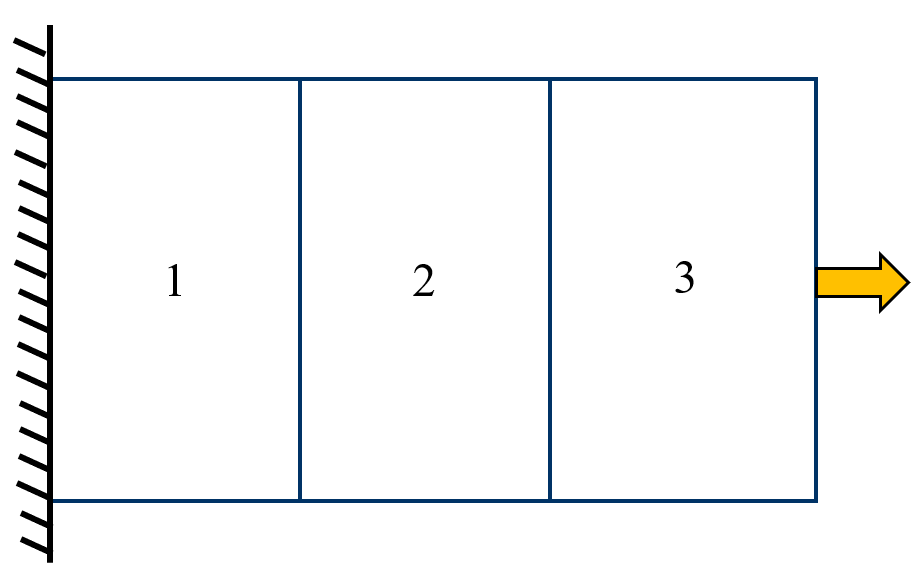} \label{model1}} 
        \subfigure[Model 2]{\includegraphics[scale=0.3]{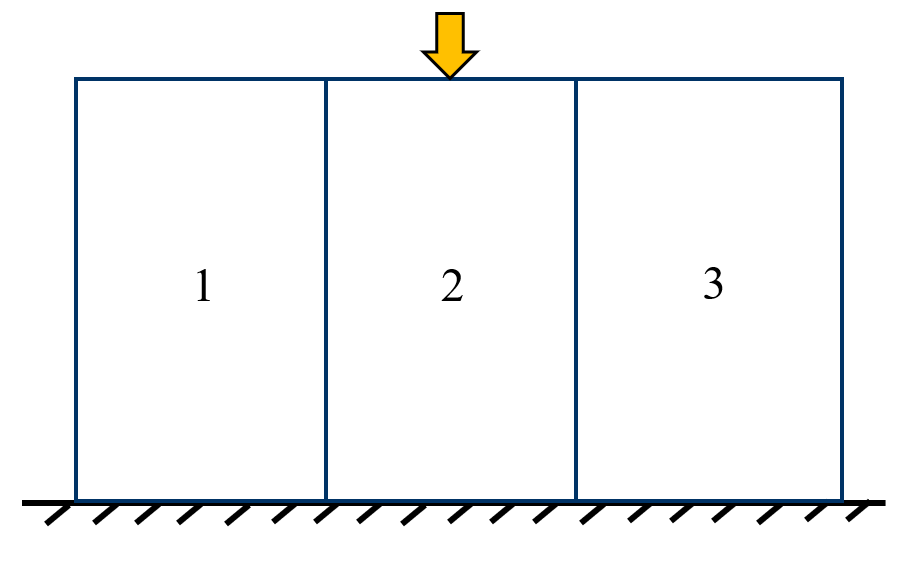} \label{model2}} 
	\subfigure[Model 3]{\includegraphics[scale=0.3]{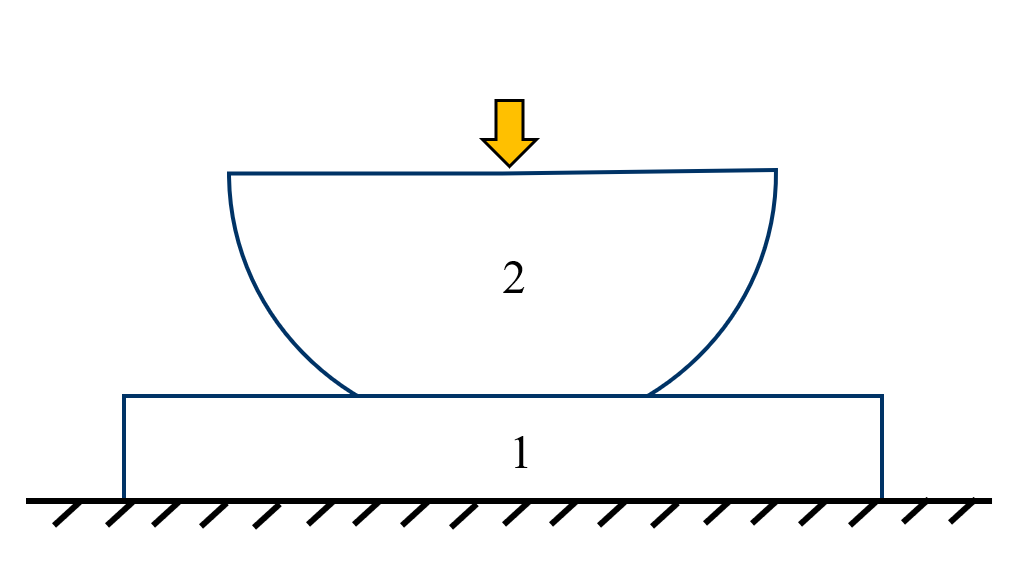} \label{model3}} 
	\caption{Tied contact example models.}
	\label{model}
\end{figure}

\subsection{Algorithms setup}
\label{subsect:alg-setup}

The condensed elimination algorithm is implemented
by using the matrix and vector data structures provided by PETSc~\cite{PETSCbalay2017petsc}.
Besides, the linear systems are solved by calling iterative methods and preconditioners provided by PETSc.
PETSc offers several Krylov subspace iterative methods, including CG, GMRES, BiCGStab, and MINRES, 
as well as some preconditioning methods like domain decomposition, 
algebraic multigrid, etc.
Additionally, some third party solver libraries, such as HYPRE~\cite{HYPREfalgout2002hypre},  
MUMPS~\cite{MUMPSamestoy2000mumps} can be called as preconditioning solvers in PETSc.

The GMRES, BiCGStab, and GCR methods are used as iterative algorithms for 
solving the original saddle point system before elimination.
After elimination, the matrix has symmetric positive definiteness property, 
and thus the CG method is used as the iterative algorithm.
The following five commonly used preconditioning methods are tested in the experiments.
\begin{itemize}
    \item JAC: Jacobi method; 
    \item BJAC: block Jacobi method;
    \item ASM: additive Schwarz method;
    \item AMG: algebraic multigrid method;
    \item SOR: successive over-relaxation method.
\end{itemize}
Besides, the SIMPLE method, which was originally designed for solving block system 
in computational fluid dynamics~\cite{patankar1983calculation}, 
has also been successfully applied in contact mechanics~\cite{AMGwiesner2021algebraic}.
Thus, it is tested as one preconditioning method
for solving the equation before elimination by considering the saddle point property.
In the implementation of the SIMPLE method,
it is needed to provide solution methods 
for solving the subsystem corresponding to 
the left-upper and right-lower blocks of the saddle point matrix, which represent the subsystem of displacement 
and Lagrange multiplier in the context of contact computation.
To solve the displacement subsystem, the AMG method is used, and the coarsening strategy is PMIS.
The ILU method is used to solve the Lagrange multiplier subsystem~\cite{AMGwiesner2021algebraic}.

When AMG method is used as a preconditioner for solving 
the linear system after elimination, 
we use HMIS coarsening method~\cite{de2006reducing}, 
and the extended+$i$ interpolation method~\cite{de2008distance}.
Furthermore, the node coarsening strategy, 
which is based on the matrix generated by the row sum norm of the eliminated matrix, is used~\cite{AMGbaker2016scalability}.


In numerical experiments, the maximum number of iterations is 2000, 
and the convergence is declared if the relative residual is less than $10^{-8}$.
In the reported results, 
{\NIT} represents the number of Krylov iterations,
and $r_{\text{rel}}^k$ represents the relative residual norm, 
that is $r_{\text{rel}}^k = \|\bvec{r}_k\|_2 / |\bvec{r}_0\|_2$.

The solution time for the saddle point system or
the condensed system is denoted as \Tsol.
In the DOFs condensation method, 
the elimination time is denoted as \Tcon,
and total time for elimination and solution time is represented by \Ttot. 
The elimination time is primarily composed of four
parts: (a) \TT, the time to construct the matrix $\bmat{T}$; (b) \TC, the time to construct the matrix $\bmat{C}$; (c) \TA, the time to generate the elimination system; and (d) \TO, other computation time.

\subsection{Solution comparison for transformed and untransformed linear systems}
\label{subsect:sol-compare}

In this subsection, we show the solution effectiveness for the original (untransformed) saddle point systems~\eqref{twobody1} and condensed (transformed) systems~\eqref{twobody2} by some preconditioned Krylov methods. 
For the tested three models, the grids on the contact surfaces are non-matching. 
Tables~\ref{tab:method1}--\ref{tab:method3} show the solution results for three models.
All tests are carried out on 28 processors.
The GMRES, BiCGStab, MINRES and GCR are used 
to solve the linear systems before elimination (saddle point system),
with GCR performs best. 
Therefore, in these tables, only the results of GCR 
are listed.
The symbol * in the tables indicates that the corresponding method failed to solve the linear system. 


\begin{table}[H]
\setlength{\tabcolsep}{3.5mm}
  \centering
  \caption{Solution effectiveness for equations of model 1.}
    \begin{tabular}{ccccccc}
    \toprule
    Equation  & DOFs    & Method   & Preconditioner   & \NIT  & $r_{\text{rel}}^k$  & \Tsol \\
    \midrule
    \multirow{7}[2]{*}{\makecell[c]{Eq.~\eqref{twobody1}}} & \multirow{7}[2]{*}{639 382} & \multirow{7}[2]{*}{GCR} & JAC   & 2000  & 7.25$ \times 10^{-2}$ & 7.06 \\
          &       &       & BJAC  & *     & *     & * \\
          &       &       & ASM   & 2000  & 6.44$ \times 10^{-5}$ & 32.02 \\
          &       &       & AMG   & 2000  & 8.54$ \times 10^{-6}$ & 25.48 \\
          &       &       & SOR   & 2000  & 6.80$ \times 10^{-2}$ & 9.29 \\
          &       &       & SIMPLE & 2000  & 9.99$ \times 10^{-1}$ & 53.39 \\
    \midrule
    \multirow{6}[2]{*}{\makecell[c]{Eq.~\eqref{twobody2}}} & \multirow{6}[2]{*}{635 374} & \multirow{6}[2]{*}{CG} & JAC   & 2000  & 7.54$ \times 10^{-3}$ & 5.55 \\
          &       &       & \textbf{BJAC}  & \textbf{776}   & $\bm{9.28\times 10^{-9}}$ & \textbf{10.18} \\
          &       &       & \textbf{ASM}   & \textbf{514}  & $\bm{ 8.97\times 10^{-9}}$ & \textbf{11.89} \\
          &       &       & \textbf{AMG}   & \textbf{21}    & $\bm{6.98\times 10^{-9}}$ & \textbf{1.34} \\
          &       &       & SOR   & 2000  & 1.99$ \times 10^{-4}$ & 8.31 \\
    \bottomrule
    \end{tabular}%
  \label{tab:method1}%
\end{table}%

\begin{table}[H]
\setlength{\tabcolsep}{3.5mm}
  \centering
  \caption{Solution effectiveness for equations of model 2.}
    \begin{tabular}{ccccccc}
    \toprule
    Equation  & DOFs    & Method   & Preconditioner   & \NIT  & $r_{\text{rel}}^k$  & \Tsol \\
    \midrule
    \multirow{7}[2]{*}{\makecell[c]{Eq.~\eqref{twobody1}}} & \multirow{7}[2]{*}{639 382} & \multirow{7}[2]{*}{GCR} & JAC   & 2000  & 8.77$ \times 10^{-2}$ & 6.62 \\
          &       &       & BJAC  & *     & *     & * \\
          &       &       & ASM   & 2000  & $4.93 \times 10^{-8}$ & 30.84 \\
          &       &       & AMG   & 2000  & 6.49$ \times 10^{-7}$ & 25.74 \\
          &       &       & SOR   & 2000  & 5.69$ \times 10^{-2}$ & 9.25 \\
          &       &       & \textbf{SIMPLE} & \textbf{261}   & $ \bm{9.47\times 10^{-9}}$ & \textbf{8.17} \\
    \midrule
    \multirow{6}[2]{*}{\makecell[c]{Eq.~\eqref{twobody2}}} & \multirow{6}[2]{*}{635 374} & \multirow{6}[2]{*}{CG} & JAC   & 2000  & 6.21$ \times 10^{-2}$ & 5.93 \\
          &       &       & \textbf{BJAC}  & \textbf{728}   & $\bm{9.75\times 10^{-9}}$ & \textbf{9.54} \\
          &       &       & \textbf{ASM}   & \textbf{489}  & $\bm{9.41\times 10^{-9}}$ & \textbf{11.07} \\
          &       &       & \textbf{AMG}   & \textbf{20}    & $\bm{3.97\times 10^{-9}}$ & \textbf{1.39} \\
          &       &       & SOR   & 2000  & 4.51$ \times 10^{-4}$ & 8.21 \\
    \bottomrule
    \end{tabular}%
  \label{tab:method2}%
\end{table}%

\begin{table}[H]
\setlength{\tabcolsep}{3.5mm}
  \centering
  \caption{Solution effectiveness for equations of model 3.}
    \begin{tabular}{ccccccc}
    \toprule
    Equation  & DOFs    & Method   & Preconditioner   & \NIT  & $r_{\text{rel}}^k$  & \Tsol \\
    \midrule
    \multirow{7}[2]{*}{\makecell[c]{Eq.~\eqref{twobody1}}} & \multirow{7}[2]{*}{712 322} & \multirow{7}[2]{*}{GCR} & JAC   & 2000  & 7.05$ \times 10^{-2}$ & 7.87 \\
          &       &       & BJAC  & *     & *     & *  \\
          &       &       & ASM   & 2000  & 7.78$ \times 10^{-7}$ & 33.14 \\
          &       &       & AMG   & 2000  & 6.51$ \times 10^{-1}$ & 28.64 \\
          &       &       & SOR   & 2000  & 6.44$ \times 10^{-2}$ & 10.32 \\
          &       &       & SIMPLE & 2000  & 9.99$ \times 10^{-1}$ & 60.9 \\
    \midrule
    \multirow{6}[2]{*}{\makecell[c]{Eq.~\eqref{twobody2}}} & \multirow{6}[2]{*}{710 718} & \multirow{6}[2]{*}{CG} & JAC   & 2000  & 2.07$ \times 10^{-2}$ & 5.36 \\
          &       &       & \textbf{BJAC}  & \textbf{737}   & $\bm{9.92\times 10^{-9}}$ & \textbf{10.43} \\
          &       &       & \textbf{ASM}  & \textbf{452}   & $\bm{8.43\times 10^{-9}}$ & \textbf{11.69} \\
          &       &       & \textbf{AMG}   & \textbf{22}    & $\bm{3.86\times 10^{-9}}$ & \textbf{1.18} \\
          &       &       & SOR   & 2000  & 3.51$ \times 10^{-4}$ & 8.34 \\
    \bottomrule
    \end{tabular}%
  \label{tab:method3}%
\end{table}%

The results in Tables~\ref{tab:method1}--\ref{tab:method3} indicate that, 
GCR does not converge within 2000 iterations when the first five preconditioning methods
are used to solve the original saddle point linear systems.
The BJAC cannot be used due to zero elements in the diagonal block of the saddle point system. 
The SIMPLE preconditioned CGR method can successfully solve the saddle point system of model 2, but it 
fails to solve the systems of models 1 and 3.
In each iteration step of the SIMPLE method, 
a linear system for displacement with respect to the upper-left block must be solved. 
Theoretically, the SIMPLE method is only applicable when the upper-left block is nonsingular. 
However, for the contact problem, the stiffness matrix $\bmat{K}$ of the upper-left block is nonsingular only under certain circumstances. 
For model 2, the fixed boundary conditions are applied to each contact body, and the corresponding stiffness matrix is positive definite. 
However, for models 1 and 3, the stiffness matrix is semi-positive definite. This will make the SIMPLE preconditioning deteriorate.  
In short, for all cases of solving the saddle point system, 
only the SIMPLE preconditioned GCR method succeeds 
for model 2, with the solution time being 8.17s.  


The results of the transformed system show that, by using the BJAC, ASM, or AMG preconditioning method, CG can converge for 
all cases of the three test models,
with AMG method performing best. 
For example, the GCR method with SIMPLE preconditioning requires 261 iterations and takes a total of 8.17s to solve the saddle point system of model 2. 
However, by employing AMG preconditioning method, CG can solve the transfored (by DOFs elimination) linear system with 20 iterations, costing 1.39s.
As a result, the condensed elimination approach can greatly reduce the difficulty of solving the linear system.


The cost of the DOFs condensation method contains 
the elimination and equation solution processes.
To show the specific cost for these processes,
Table~\ref{tab: time} is presented about 
the elimination time \Tcon, solution time \Tsol, 
total time \Ttot, and the proportion of elimination time for three models
when the number of processors is 28.


\begin{table}[H]
\setlength{\tabcolsep}{6.5mm}
  \centering
  \caption{Distribution of the elimination time for three models.}
    \begin{tabular}{cccccr}
    \toprule
    Model   & Preconditioner   & \Tcon  & \Tsol  & \Ttot  & \Tcon / \Ttot \\
    \midrule
    \multirow{3}[1]{*}{model 1} & BJAC  & 0.88  & 10.18 &  11.06 & 7.96\% \\
          & ASM   & 0.88  & 11.89 &  12.77 &  7.83\% \\
          & AMG   & 0.88  &  1.34 &   2.22 & 39.64\% \\
    \midrule
    \multirow{3}[2]{*}{model 2} & BJAC  & 0.91  &  9.54 &  10.45 & 8.71\% \\
          & ASM   & 0.91  & 11.07 &  11.98 &  7.60\% \\
          & AMG   & 0.91  &  1.39 &   2.30 & 39.57\% \\
    \midrule
    \multirow{3}[2]{*}{model 3} & BJAC  & 0.90  & 10.43 &  11.33 & 7.94\% \\
          & ASM   & 0.90  & 11.69 &  12.59 &  7.15\% \\
          & AMG   & 0.90  &  1.18 &   2.08 & 43.27\% \\
    \bottomrule
    \end{tabular}%
  \label{tab: time}%
\end{table}%

From Table~\ref{tab: time}, one can find that the elimination time is less than the solution time for all cases. The elimination time is less than 1 second 
for three models. The equation solution time is 
greater than 11 seconds for ASM preconditioning, 
9.5 seconds for BJAC preconditioning, and
1.18 seconds for AMG preconditioning.



\subsection{Algorithm scalability}
\label{subsect:alg-scalability}

By the results in Subsection~\ref{subsect:sol-compare},
the CG method with AMG preconditioning is the best solver for solving the problems. 
To show the algorithmic scalability of the method 
for different sizes of the problem, 
Table~\ref{tab: size} is presented.
In this table, four cases of DOFs are considered for each model. 
The total iteration numbers and solution time with 28 MPI processors are listed
for all cases.
The results show that the iteration number is 
relatively stable with the variation of the numbers of DOFs. At the same time, it can be observed that
the increase of solution time is less than the increase 
of the number of DOFs. Therefore,
for solving the transformed (DOFs condensed) system,
AMG preconditioned CG method shows good 
algorithm scalability.
This further shows that the DOFs condensation method is effective for different sizes of problems.


\begin{table}[H]
\setlength{\tabcolsep}{5.5mm}
  \centering
  \caption{Iteration numbers and solution time of the AMG preconditioned CG method for solving the transformed systems with different DOFs.}
    \begin{tabular}{crccccc}
    \toprule
    Model & \multicolumn{1}{c}{DOFs} & \Tcon  & \Tsol  & \Ttot  & \NIT   & $r_{\text{rel}}^k$ \\
    \midrule
    \multirow{4}[1]{*}{model 1} & 5 661 334 & 5.50  & 9.85  & 15.35  & 24    & 6.54$ \times 10^{-9}$ \\
          & 10 078 520 & 7.74  & 14.24  & 21.98  & 23    & 9.32$ \times 10^{-9}$ \\
          & 22 483 048 & 14.52  & 29.65  & 44.17  & 26    & 7.53$ \times 10^{-9}$ \\
          & 40 028 566 & 26.21  & 64.56  & 90.77  & 25    & 8.54$ \times 10^{-9}$ \\
    \midrule
    \multirow{4}[2]{*}{model 2} & 5 661 334 & 5.18  & 9.22  & 14.40  & 22    & 8.48$ \times 10^{-9}$ \\
          & 10 078 520 & 7.82  & 14.35  & 22.17  & 22    & 5.32$ \times 10^{-9}$ \\
          & 22 483 048 & 14.06  & 30.08  & 44.13  & 25    & 9.15$ \times 10^{-9}$ \\
          & 40 028 566 & 22.55  & 58.42  & 80.97  & 24    & 5.37$ \times 10^{-9}$ \\
    \midrule
    \multirow{4}[2]{*}{model 3} & 5 351 164 & 7.11  & 8.20  & 15.31  & 24    & 5.91$ \times 10^{-9}$ \\
          & 11 369 428 & 12.50  & 15.23  & 27.73  & 23    & 8.20$ \times 10^{-9}$ \\
          & 22 465 866 & 20.04  & 29.35  & 49.39  & 25    & 3.64$ \times 10^{-9}$ \\
          & 40 729 666 & 31.56  & 58.24  & 89.80  & 25    & 5.32$ \times 10^{-9}$ \\
    \bottomrule
    \end{tabular}%
  \label{tab: size}%
\end{table}%

\subsection{Parallel scalability}
\label{subsect:para-scalability}

To show the parallel scalability of the 
DOFs condensation combined with AMG preconditioned
CG method, Table~\ref{tab: np} is presented for 
model 1 with a fixed degree of freedom of 22.48 million.
By increasing the number of processors, the
solution time, the speedups (denoted by $S_{p}$) and parallel efficiency (denoted by $\eta$) 
is listed in the table.
At the same time, the iteration numbers 
are listed.


\begin{table}[H]
\setlength{\tabcolsep}{5.2mm}
  \centering
  \caption{Parallel performance of the DOFs condensation and AMG preconditioned CG method for solving model 1  (22 483 048 DOFs).}
    \begin{tabular}{cccccccc}
    \toprule
    Processors    & \Tcon  & \Tsol  & \Ttot  & \NIT  & $r_{\text{rel}}^k$  & $S_p$   & $\eta$(\%) \\
    \midrule
    14    & 18.17  & 64.16  & 82.33  & 37    & 4.44$ \times 10^{-9}$ & 1.00  & 100 \\
    28    & 14.45  & 36.45  & 50.91  & 37    & 3.77$ \times 10^{-9}$ & 1.62  & 80.86 \\
    56    & 13.28  & 26.61  & 39.88  & 38    & 3.16$ \times 10^{-9}$ & 2.06  & 51.61 \\
    112   & 19.73  & 30.17  & 49.90  & 38    & 8.71$ \times 10^{-9}$ & 1.65  & 20.62 \\
    \bottomrule
    \end{tabular}%
  \label{tab: np}%
\end{table}%

It can be observed from Table~\ref{tab: np} that 
when the number of processors increases from 14 to 56,
the total solving time decreases from 82.33s to 39.88s,
and the speedup varies from 1 to 2.06,
the corresponding parallel efficiency decreases from 100\%
to 51.61\%.
When the number of processors further increases to 112,
the total solving time increases, and 
the corresponding parallel efficiency decreases to 20.65\%.
This shows that the algorithm can be scaled to 56 processors. 
Because of the total size of the 
problem is fixed, when 112 processors are used
to solve the problem, the cost of communication 
increases too much such that 
the total solution increases.


%

Lastly, Fig.~\ref{fig: displace} is given to show the deformation of 
models 1 and 3, where the color represents the magnitude 
of the displacement.
From Fig.~\ref{fig: three-displace}, one can see that almost the whole body has some deformation, with the part near right boundary showing
the largest deformation. This can be reflected by the value of the magnitude of the displacement. 

From Fig.~\ref{fig: two-displace} one can see that, for model 3, 
the upper body has some deformation, while the deformation of the lower body is not obvious.
The largest value of the magnitude of displacement 
is near the upper side, 
where corresponding the 
relatively largest deformation.



\begin{figure}[H]
	\centering
	\subfigure[Model 1]{\includegraphics[scale=0.18]{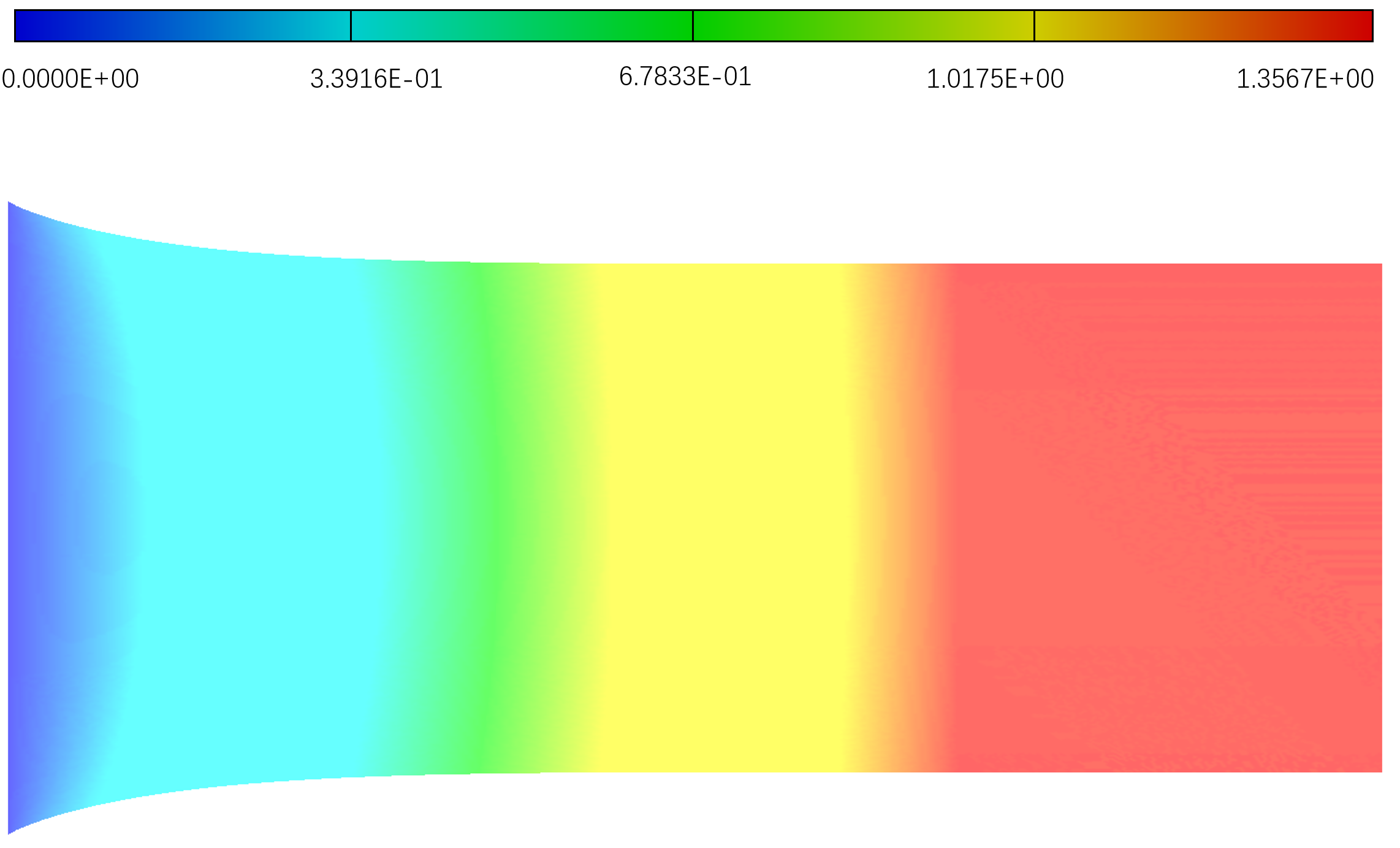} \label{fig: three-displace}} 
	\hspace{0.5in}
        \subfigure[Model 3]{\includegraphics[scale=0.18]{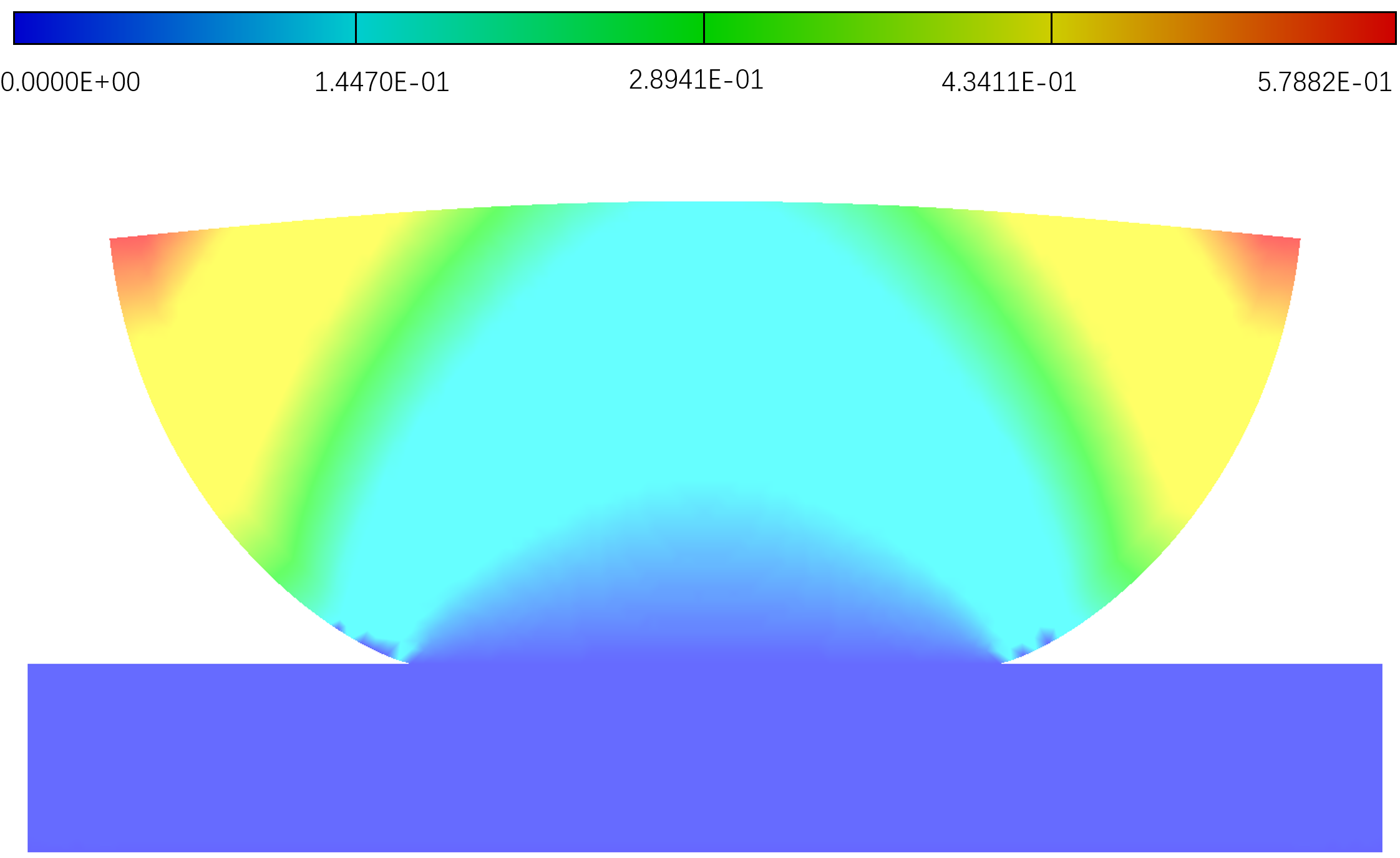} \label{fig: two-displace}} 
	\caption{Deformation of example models.}
	\label{fig: displace}
\end{figure}

\section{Conclusion and Remarks}
\label{sect:remark}

It is crucial to solve the linear system in contact computation.
Particularly, when the contact constraint is enforced by the Lagrange method, 
the obtained saddle point system is difficult to solve 
because the coefficient matrix is indefinite.
In this paper, by using the characteristics of the contact computation,
an efficient solution strategy is proposed for solving 
the linear system of the two-dimensional
tied contact applications. 
The strategy consists of two steps. First, the saddle point system is transformed into a symmetric positive definite system 
by eliminating the degrees of freedom of the Lagrange multiplier and 
displacement on the slave surfaces. Then the obtained linear 
system is solved by a preconditioned CG method.
The key of the transformation is based on 
one of the mortar matrix's tridiagonal block property.

Numerical results for three contact models show that the proposed 
strategy is effective for solving the linear systems in contact computation. 
The saddle point system is difficult to solve by using 
the commonly used methods. The transformed system
can be solved by the preconditioned CG method. 
Among the tested methods, the CG with AMG preconditioning performs the best, achieving fast convergence with fewer iterations. 
The time spent on the elimination process is generally 
less than half of the total time.
Furthermore, 
the solution strategy is algorithmic scalable measured by the
iteration of the AMG preconditioned CG method. 

For the proposed DOFs condensation based method,
it is needed to do further research on two aspects:
first, the cost of elimination is dominated by 
the product of three matrices. More research should be done to reduce 
this cost. Second, in this paper, we only consider the two-dimensional case.
For three-dimensional contact problems, some specific elimination strategies should 
be designed.



\bibliographystyle{elsarticle-num} 
\bibliography{DOFsCondensationAlgorithm}





\end{document}